\documentclass[11pt]{amsart}
\usepackage{amsmath}
\usepackage{amssymb}
\usepackage{verbatim}
\usepackage[linktocpage]{hyperref}
\usepackage{genyoungtabtikz}

\newtheorem{df}{Definition}[section]
\newtheorem{thm}[df]{Theorem}

\newtheorem{prop}[df]{Proposition}
\newtheorem{rem}[df]{Remark}

\newtheorem{lem}[df]{Lemma}
\newtheorem{cla}[df]{Claim}
\newtheorem{cor}[df]{Corollary}

\newcommand{\eq}{eqnarray*}

\newcommand{\ti}{\tilde}

\title[\resizebox{4.5in}{!}{Nilpotent Higgs bundles and the Hodge metric on the Calabi-Yau moduli}]{Nilpotent Higgs bundles and the Hodge metric on the Calabi-Yau moduli}
\author{Qiongling Li}

\address{Qiongling Li\\
Chern Institute of Mathematics and LPMC\\
Nankai University\\
No. 94 Weijinlu Nankai District\\
Tianjin\\
P.R.China 300071}

\email{qiongling.li@nankai.edu.cn}
\date{}

\begin{document}
\maketitle
\begin{abstract}
We study an algebraic inequality for nilpotent matrices and show some interesting geometric applications: (i) obtaining topological information for nilpotent polystable Higgs bundles over a compact Riemann surface; (ii) obtaining a sharp upper bound of the holomorphic sectional curvatures of the period domain and the Hodge metric on the Calabi-Yau moduli. 
\end{abstract}
\section{Introduction}
We study an algebraic function on orbits of nilpotent matrices in this paper and show how it gives geometric applications in the following settings by relating the algebraic function with the curvature formula on homogeneous spaces. 
  
First, given a compact Riemann surface $\Sigma$ of genus at least $2$, the non-abelian correspondence, developed by Hitchin \cite{Hitchin87}, Simpson \cite{Simpson88}, Corlette \cite{Corlette}, and Donaldson \cite{Donaldson}, says that the space of conjugacy classes of reductive representations $\rho:\pi_1(\Sigma)\rightarrow SL(n,\mathbb C)$ is homeomorphic to the moduli space of polystable $SL(n,\mathbb C)$-Higgs bundles over $\Sigma$. A Higgs bundle $(E,\phi)$ is said to be \textit{nilpotent} if $\phi$ satisfies $\phi^{\otimes n}=0$. We will obtain topological properties of corresponding representations for nilpotent polystable $SL(n,\mathbb C)$-Higgs bundles over $\Sigma$. 

Secondly,  for a deformation family of polarized K\"ahler manifolds, by considering the variation of primitive cohomology groups, it gives rise to a period map $\mathcal P$ from the base manifold to the period domain $\mathcal D$. By Griffiths \cite{Griffiths}, the period map is holomorphic and the image of the tangential map $T\mathcal P$ lies in the horizontal distribution $T^h\mathcal D$ of $\mathcal D$, which has tangent vector as a special type of nilpotent matrices. We will give a sharp upper bound of holomorphic sectional curvatures of $T^h\mathcal D$. As an application, we show the holomorphic sectional curvature bound of the Hodge metric, defined by Lu \cite{Lu}, on the universal Calabi-Yau moduli. 

\subsection{Riemann surface}
For a compact Riemann surface $\Sigma$ of genus at least $2$, let $K$ denote the canonical bundle of $\Sigma$. The uniformization gives rises to a representation $\hat j_{\Sigma}:\pi_1(\Sigma)\rightarrow PSL(2,\mathbb{R})$ which can be lifted to $j_{\Sigma}:\pi_1(\Sigma)\rightarrow SL(2,\mathbb{R})$. 

We first stratify the nilpotent cone of the moduli space of Higgs bundles according to their Jordan types. As in Schiffmann \cite{Schiffmann} Section 3, for a nilpotent Higgs bundle $(E,\phi)$ over $\Sigma$, one can define its \textit{Jordan type} $J(E,\phi)\in\mathcal P_n$, where $\mathcal P_n$ is the space of all partitions of $n$, as follows: 
$$J(E,\phi)=(\lambda_1,\lambda_2,\cdots,\lambda_n)\in \mathcal P_n,$$
where $F_i$ is a holomorphic subbundle of $E$ generated by $\ker(\phi^i)$ and 
$rank(F_i)-rank(F_{i+1})=\lambda_i+\cdots+\lambda_n,$ for $i=1,\cdots,n.$ There is a natural partial order in $\mathcal P_n$, thus we can say a nilpotent Higgs bundle is of Jordan type at most $\lambda$.

Given a partition $\lambda\in \mathcal{P}_n$, using the unique irreducible representation $\tau_r: SL(2, \mathbb C)\rightarrow SL(r,\mathbb C)$, one can define a natural representation $\tau_{\lambda}: SL(2,\mathbb C)\rightarrow SL(n,\mathbb C).$ Let $\mathbb P$ be the natural projection from $SL(n,\mathbb C)$ to $PSL(n,\mathbb C)$. Given a $SL(n,\mathbb C)$-invariant metric on $SL(n,\mathbb C)/SU(n)$ induced by $(X,Y)=2tr(XY)$ for $X,Y\in sl(n,\mathbb C)$, the translation length of $\gamma$ with respect to a representation $\rho:\pi_1(\Sigma)\rightarrow SL(n,\mathbb{C})$ is defined by 
$$l_{\rho}(\gamma):=\inf_{x\in SL(n,\mathbb C)/SU(n)}d(x,\rho(\gamma)x),$$ where $d(\cdot,\cdot)$ is the distance induced by the Riemannian metric.
\begin{thm}\label{NilpotentDominationIntro}(Theorem \ref{NilpotentDomination})
Suppose a nilpotent polystable $SL(n,\mathbb{C})$-Higgs bundle $(E,\phi)$ over $\Sigma$ is of Jordan type at most $\lambda\in P_n$. Let $\rho:\pi_1(\Sigma)\rightarrow SL(n,\mathbb C)$ be its associated representation. Then there exists a positive constant $\alpha<1$ such that the translation length spectrum satisfies $$l_\rho\leq \alpha\cdot l_{\tau_{\lambda}\circ j_\Sigma},$$
unless $\mathbb P(\rho)=\mathbb P(\tau_{\lambda}\circ j_\Sigma)$, in which case,
\[(E,\phi)=\oplus_{i=1}^r(E_i,\phi_i), \quad (E_i,\phi_i)=Sym^{\lambda_i-1}(K^{\frac{1}{2}}\oplus K^{-\frac{1}{2}}, \begin{pmatrix}0&0\\1&0\end{pmatrix})\otimes (V_i, 0),\]  
where $\lambda=(\lambda_1^{k_1},\cdots,\lambda_r^{k_r})$ and for each $i$, $V_i$ is a polystable holomorphic vector bundle of rank $k_i$ satisfying $\prod_{i=1}^r\det(V_i)^{\lambda_i}=\mathcal O$. 
\end{thm}

Note that $(n)$ is maximal in $\mathcal P_n$. As a direct corollary of Theorem \ref{NilpotentDominationIntro}, 
\begin{cor}\label{RoughDomination}
For any nilpotent polystable $SL(n,\mathbb{C})$-Higgs bundle over $\Sigma$, the corresponding representation $\rho$  satisifies $l_\rho\leq \alpha\cdot l_{\tau_{n}\circ j_\Sigma}$ for some positive constant $\alpha<1$, unless $\mathbb P(\rho)=\mathbb P(\tau_n\circ j_{\Sigma})$ .
\end{cor}
\begin{rem}
The nilpotent cone consisting of nilpotent Higgs bundles is the fiber at $0$ of the Hitchin fibration from the moduli space of Higgs bundles to the vector space $\bigoplus\limits_{j=2}^nH^0(\Sigma, K^i)$. In \cite{DaiLiHitchinFiber}, Dai and the author generalize Corollary \ref{RoughDomination} to  Higgs bundles in Hitchin fibers which contain $n$-Fuchsian Higgs bundles. 
\end{rem}

The entropy of a representation $\rho: \pi_1(\Sigma)\rightarrow SL(n,\mathbb C)$ is defined as \begin{\eq}
h(\rho):=\underset{R\rightarrow \infty}{\text{lim sup }}\frac{\text{log}(\#\{\gamma\in \pi_1(\Sigma)|l_{\rho}(\gamma)\leq R\})}{R}.
\end{\eq} 
An immediate corollary on the entropy of Corollary \ref{RoughDomination} is as follows.
\begin{cor}
For a nilpotent polystable $SL(n,\mathbb{C})$-Higgs bundle $(E,\phi)$ over $\Sigma$, the entropy of its associated representation $\rho$ satisfies if it is finite, then 
$h(\rho)\geq \sqrt{\frac{6}{n(n^2-1)}}.$  

Equality holds if and only if $(E,\phi)$ is the tensor product of $Sym^{n-1}(K^{\frac{1}{2}}\oplus K^{-\frac{1}{2}}, \begin{pmatrix}0&0\\1&0\end{pmatrix})$ with $(L,0)$ where $L$ is a holomorphic line bundle satisfying $L^n=\mathcal O$, in which case, $\mathbb P(\rho)=\mathbb P(\tau_n\circ j_{\Sigma})$. 
\end{cor}
\begin{rem}
Hitchin \cite{Hitchin92} defines a section of the Hitchin fibration whose image forms a connected component (the Hitchin component) of the $SL(n,\mathbb{R})$-Higgs bundle moduli space. The corresponding representations are called Hitchin representations. Potrie and Sambarino \cite{HitchinEntropy} showed for any Hitchin representation $\rho$, $h(\rho)\leq\sqrt{\frac{6}{n(n^2-1)}}$ (with an appropriate normalization) and the equality holds only if $\mathbb P(\rho)=\mathbb P(\tau_n\circ j_{\Sigma})$ for some Riemann surface $\Sigma$. We can see that the nilpotent cone possesses an opposite behavior of the Hitchin section.  
\end{rem}

\subsection{Period Domain and Calabi-Yau moduli} 
For a polarized K\"ahler manifold $(X, \omega)$, a deformation family of polarized K\"ahler manifolds $X_t$ over $t\in B$ from $(X,\omega)$, by associating to each fiber $X_t$ the Hodge decomposition of its $k$-th primitive cohomology, we get the period map $\mathcal P: S\rightarrow \Gamma\backslash\mathcal D$, where $\Gamma$ is the monodromy group and $\mathcal D=\mathcal D(H, Q, k, \{h^{p,q}\})$ is the period domian classifying all Hodge structures of weight $k$ with fixed dimension $h^{p,q}$ of $H^{p,q}$, polarized by $Q$. It is a homogeneous space $G/V$ endowed with a canonical $G$-invariant metric which is unique up to a constant multiple. The period map is holomorphic and its tangential map has image in the horizontal distribution $T^h\mathcal D$, shown by Griffiths \cite{Griffiths}. Moreover, Griffiths and Schmid \cite{GriffithsSchmid} show that the horizontal distribution $T^h\mathcal D$ always has negative holomorphic sectional curvature. 

Here we give an effective estimate of the holomorphic sectional curvature of $T^h\mathcal D$ once we fix a choice of metric $h$ on $\mathcal D$ as in Equation (\ref{MetricPeriodDomain}). The estimate only depends on the composition $(h^{k,0},h^{k-1,1},\cdots, h^{0,k})$ of ${\sum_{p=0}^kh^{p,q}}$.

\begin{thm}\label{PeriodDomainCurvatureIntro}(Theorem \ref{PeriodDomainCurvature})
The $G$-invariant Hermitian metric $h$ on $\mathcal D=\mathcal D(H, Q, k, \{h^{p,q}\})$ has holomorphic sectional curvature in the direction $\xi\in T^h\mathcal D$ satisfying 
\[K(\xi)\leq -C_{\mathcal R^t}.\]
where $\mathcal R^t$ is the conjugate partition of $\mathcal R=(h^{k,0},h^{k-1,1},\cdots, h^{0,k})$ as defined in Section \ref{Composition} and the constant $C_{\mathcal R^t}$ is defined in Equation (\ref{ConstantC}). 

Moreover, the equality can be achieved in some direction $\xi$.
\end{thm}

A polarized Calabi-Yau $m$-manifold is a pair $(X,\omega)$ of a compact algebraic manifold $X$ of dimension $m$ with vanishing first Chern class and a K\"ahler form $\omega\in H^2(X,\mathbb{Z})$. The universal deformation space $\mathcal{M}_X$ of polarized Calabi-Yau $m$-manifolds is smooth, shown by Tian \cite{Tian}. The tangent space $T_{X'}\mathcal{M}_X$ of $\mathcal{M}_X$ at $X'$ can be identified with $H^1(X', T_{X'})$. Denote $n=\dim_{\mathbb C} H^1(X,T_X)= h^{1,m-1}$, where $h^{p,m-p}$ is the dimension of the $(p,m-p)$-primitive cohomology group of $(X,\omega)$. So $h^{m,0}=h^{0,m}=1, h^{m-1,1}=h^{1,m-1}=n.$

We then have the associated period map $\mathcal P: \mathcal M_X\rightarrow \Gamma\backslash\mathcal D$. The pullback metric $\mathcal P^*h$ on $\mathcal M_X$ by the period map is called the Hodge metric, first defined by Lu \cite{Lu}. There is a close relation between the Weil-Petersson metric and the Hodge metric for deformation space of Calabi-Yau manifolds as stated in Proposition \ref{RelationTwoMetrics}.  

An important application of Theorem \ref{PeriodDomainCurvatureIntro} is the following estimate of the holomorphic sectional curvatures of the Hodge metric on $\mathcal M_X$.
\begin{thm}(see Section \ref{CalabiYauModuli})
For a polarized Calabi-Yau $m$-manifold $(X,\omega)$, let $n$ be the dimension of the universal deformation space $\mathcal M_X$. Then the Hodge metric over $\mathcal M_X$ has its holomorphic sectional curvature bounded from above by a negative constant $$c_m=-C_{\mathcal R^t},$$ where $\mathcal R^t$ is the conjugate partition of $\mathcal R=(h^{m,0},h^{m-1,1},\cdots, h^{0,m})$ as defined in Section \ref{Composition} and the constant $C_{\mathcal R^t}$ is defined in Equation (\ref{ConstantC}). 

In particular, \\
(1) $c_3=-\frac{2}{n+9}$.\\
(2) $c_4=-\frac{1}{2(\min\{a,n\}+4)}$ for $a=h^{2,2}$.\\
(3) $c_5=-\frac{2}{(9\min\{a,n\}+a+25)}$ for $a=h^{3,2}$.
\end{thm}
\begin{rem}
(1) By the Schwarz-Yau lemma, an optimal bound of holomorphic sectional curvatures will give estimates of the Weil-Petersson metric on a complete algebraic curve inside the moduli space $\mathcal M_X$, e.g. Theorem 6.1 in Lu \cite{Lu}. As additional applications, such estimates can give refined Arakelov-type inequalities, e.g. Theorem 0.6 in Liu-Xia \cite{LiuXia}.\\
(2) For $m=3$, Lu in \cite{Lu} gave the upper bound $-\frac{1}{(\sqrt{n}+1)^2+1}$. Here we improve to the upper bound $-\frac{2}{n+9}$ and show that it is the optimal upper bound in the algebraic sense.\\
(3) For $m=4$, Lu-Sun \cite{LuSun} in Theorem 1.2 showed the upper bound is $-\frac{1}{2(n+4)}$. Here we obtain a refined upper bound by replacing $n$ by $\min\{a, n\}$ where $a=h^{2,2}$.\\
(4) For $m=5$ and higher, our estimates are new.
\end{rem}

\subsection*{Organization of the paper} In Section \ref{AlgebraicInequalities}, we prove some algebraic inequalities for nilpotent matrices and matrices of type $\mathcal R\in \mathcal C_n$. Then we apply the algebraic inequalities to nilpotent Higgs bundles over complex manifolds in Section \ref{GeneralComplexManifolds}. In Section \ref{RiemannSurface}, we obtain topological information for nilpotent Higgs bundles over Riemann surfaces. In Section \ref{PeriodDomainModuli}, we discuss how the algebraic estimate for matrices of type $\mathcal R\in C_n$ applies to the study of curvature estimates of the period domain and the deformation space of Calabi-Yau manifolds. 

\subsection*{Acknowledgement}This paper is inspired by the paper of Xu Wang \cite{XuWang}. The author wants to thank Nicolas Tholozan for inspiring and helpful discussions and Richard Wentworth for referring to her the study on the classifying space of Hodge structures. The author acknowledges support from Nankai Zhide Foundation.

\section{Algebraic Inequalities}\label{AlgebraicInequalities}
We first review some basic knowledge on the relation between partitions and nilpotent orbits, one can refer to Section 3.1 in the book of Collingwood and McGovern \cite{CollingwoodMcGovern}. A \textbf{partition} of $n$ is a weakly decreasing array $(\lambda_i)$ of integers  $\lambda_1\geq \lambda_2\geq \cdots\geq \lambda_n$
satisfying $\lambda_i\geq 0, \sum\limits_{p=1}^n\lambda_p=n$. Sometimes we omit the zeros to represent $\lambda$. Denote by $\mathcal{P}_n$ the space of all partitions of $n$. We sometimes use exponential notation for partitions. For example, $(4,3^2,2,1^2)=(4,3,3,2,1,1)$.

Let $\lambda\in \mathcal P_n$ and $\mu\in \mathcal P_{m}$, then define $\lambda\cup\mu$ to be the partition of $m+n$ whose parts are those of $\lambda$ and $\mu$, arranged in non-increasing order. 

The space $\mathcal{P}_n$ has a natural partial ordering, called the dominance ordering. If $\lambda,\mu$ are two distinct partitions, the partition $\lambda$ is said to \textbf{dominate} $\mu$ ($\lambda\geq\mu$) if for all $p\leq n$, $\sum\limits_{i=1}^p\lambda_i\geq\sum\limits_{i=1}^p\mu_i.$ For example, in case $n=4$, $(4)>(3,1)>(2,2)> (2,1,1)> (1,1,1,1)$.

\begin{lem}\label{DisjointUnionDominance}
Given $\lambda_1, \mu_1\in \mathcal P_n$ and $\lambda_2, \mu_2\in \mathcal P_m$ satisfying $\lambda_1\leq \lambda_2$ and $\mu_1\leq \mu_2$, then as elements of $\mathcal P_{n+m}$,  $\lambda_1\cup\lambda_2\leq \mu_1\cup\mu_2$.
\end{lem}

Given a positive integer $i$, an elementary Jordan block of type $i$ is an $i\times i$-matrix as follows:
\begin{\eq}
J_i=\begin{pmatrix}
0&0&0&\cdots&0&0\\
r_1&0&0&\cdots&0&0\\
0&r_2&0&\cdots&0&0\\
\vdots&\vdots&\vdots&\ddots&\vdots&\vdots\\
0&0&0&\cdots&0&0\\
0&0&0&\cdots&r_{i-1}&0\\
\end{pmatrix}, \end{\eq} where $r_p=\sqrt{p(i-p)}$.
 Given a partition $\lambda\in \mathcal P(n)$, let $k$ be the largest index such that $\lambda_k>0$ and define $X^{\lambda}=diag(J_{\lambda_1},\cdots, J_{\lambda_k})$. The matrix $X^{\lambda}$ is clearly a nilpotent element of $sl(n,\mathbb C)$. Note that $\{X^{\lambda}, (X^{\lambda})^*, [X^{\lambda}, (X^{\lambda})^*]\}$ forms a $sl(2,\mathbb C)$-triple.

 Define the nilpotent orbit $\mathcal O^{\lambda}=SL(n, \mathbb C)\cdot X^{\lambda}.$ There is a one-to-one correspondence between the set of nilpotent orbits in $sl(n,\mathbb C)$ and the set $\mathcal P(n)$ of partitions of $n$. The correspondence sends a nilpotent matrix $X$ to the partition determined by the block sizes in its Jordan normal form.

\subsection{Nilpotent orbits and the moment map}

Recall the Cartan decomposition of $sl(n,\mathbb C)$ is $sl(n,\mathbb C)=su(n)\oplus i\cdot su(n)$ and the Cartan involution is $\sigma(X)=-X^*$, where $X^*=\overline X^T$. Using the rescaled Killing form $B(X, Y)=tr(XY)$ on $sl(n,\mathbb R)$ and the Cartan involution, we then have an $SU(n)$-invariant Hermitian inner product on $sl(n,\mathbb C)$ by
\begin{equation}
(X, Y)=-B(X, \sigma(Y))=tr(XY^*),\quad \text{for} ~X, Y\in sl(n,\mathbb C).
\end{equation} As usual, $||X||^2$ denotes $(X,X)$.
Ness in \cite{Ness} defined a map $m: sl(n, \mathbb C)\rightarrow i\cdot su(n)$ implicitly by the equation 
\begin{equation}
(m(\xi),\eta)=\frac{1}{2||\xi||^2}(\frac{d}{dt}||Ad(exp(t\eta)\xi||^2)|_{t=0}\quad \text{for $\xi, \eta\in sl(n,\mathbb C)$},
\end{equation} which measures the change of the square norm of a vector under the adjoint action. The explicit formula of $m$ is \begin{equation}m(A)=\frac{[A, A^*]}{||A||^2}. \end{equation}
The map $m$ is invariant under scaling by $\mathbb C^*$ and hence descends to a map from $\mathbb P(sl(n,\mathbb C))$, still denote by $m$. The $SU(n)$-invariant Hermitian metric on $sl(n,\mathbb C)$ induces the natural Fubini-Study metric on the quotient space $\mathbb P(sl(n,\mathbb C))$ whose $(1,1)$-form gives a symplectic structure. Clearly the induced action of $SU(n)$ on $\mathbb P(sl(n,\mathbb C))$ preserves the symplectic structure. Observe that $m(A)$ is $SU(n)$-equivariant. Moreover, Ness in \cite{Ness} showed that $i\cdot m: sl(n,\mathbb C)\rightarrow su(n)$ is a moment map for the induced action of $SU(n)$ on $\mathbb P(sl(n,\mathbb C))$.

We consider the function $K: sl(n,\mathbb C)\rightarrow \mathbb R$ given by
\begin{equation}\label{DefinitionK}K(A)=||m(A)||^2=\frac{||[A, A^*]||^2}{||A||^4},\end{equation} the square norm of $m$.

Denote by $\mathcal N$ the space of nilpotent matrices inside $sl(n,\mathbb C)$ and by $\mathcal O_A$ the adjoint orbit of $A\in sl(n, \mathbb C)$. We will make use of the following important lemma.\begin{lem}(Theorem 6.1 and 6.2 in Ness \cite{Ness} and Lemma 2.11 in Schmid-Vilonen \cite{SchmidVilonen})\label{NessLemma}
A point $A\in \mathcal O_A$ is a critical point of the function $A\longmapsto K(A)$ if and only if there exists a real number $a$, $a<0$, such that 
\begin{equation}\label{Rigidity}
[[A, A^*], A]=aA, \quad\text{and}\quad [[A, A^*], A^*]=-aA^*.
\end{equation}
The set of critical points is non-empty and consists of a single $SU(n)\times \mathbb C^*$-orbit. Moreover, the function $K$ on $\mathcal O_A$ assumes its minimum value exactly on the critical set. 
\end{lem}
\begin{rem}
Dai and the author in \cite{DaiLiHitchinFiber} prove a generalized theorem of Lemma \ref{NessLemma} and give an independent proof of Lemma \ref{NessLemma} as a byproduct. 
\end{rem}

For each $\lambda=(\lambda_1,\cdots,\lambda_n)\in \mathcal{P}_n$, we associate a constant \begin{equation}\label{ConstantC}C_{\lambda}:=K(X^{\lambda})=\frac{4}{\sum\limits_{p=1}^n\sum\limits_{s=1}^{\lambda_p}(\lambda_p+1-2s)^2}=\frac{12}{\sum\limits_{p=1}^k\lambda_p(\lambda_p^2-1)}.\end{equation}

For a nilpotent matrix $A\in sl(n,\mathbb C)$ and $\mu\in \mathcal P_n$, we say it is of \textbf{Jordan type} $\mu$ if the block sizes of $A$'s Jordan normal form give the partition $\mu$ of $n$; we say it is of Jordan type at most $\mu$ if the block sizes of $A$'s Jordan normal form give the partition $\lambda$ where $\lambda\leq \mu$.
\begin{prop}\label{JordanInequality}
Suppose $A\in \mathcal N$ is of Jordan type at most $\lambda\in \mathcal{P}_n$, then
$$K(A)\geq C_{\lambda}, $$
and equality holds if and only $A$ is $SU(n)$-conjugate to $c\cdot X^{\lambda}$, for some constant $c\in\mathbb C^*$.
\end{prop}

\begin{proof}
Apply Lemma \ref{NessLemma} to our case that $A$ is nilpotent of Jordan type $\mu\in \mathcal{P}_n$ for some $\mu\leq \lambda$. Since $X^{\lambda}$ satisfies Equation (\ref{Rigidity}), we obtain that all the minimum points are $SU(n)$-conjugate to $c\cdot X^{\mu}$ for some constant $c\in \mathbb C^*$ and hence $K(A)\geq K(X^{\mu})=C_{\mu}.$ From the monotonicity in Lemma \ref{OrderOfPartition}, we have $C_{\mu}\geq C_{\lambda}$ and hence $K(A)\geq C_{\lambda}.$ The rigidity also follows easily.
\end{proof}
Note that among $\mathcal{P}_n$, $\lambda=(n)$ is the absolute maximum. Therefore, we have an immediate corollary of Proposition \ref{JordanInequality}.
 \begin{prop}
For every $A\in \mathcal N$, we have 
$$K(A)\geq C_{(n)}=\frac{12}{n(n^2-1)}, $$
and equality holds if and only if $A$ is $SU(n)$-conjugate to $c\cdot J_n$, for some constant $c\in \mathbb C^*$.
 \end{prop}

\subsection{Young diagram and the conjugate partition} 
One can see the reference on Page 65 in \cite{CollingwoodMcGovern}. 
Given a partition $\lambda\in\mathcal P(n)$, define a new partition $\lambda^t=(\lambda_1^t,\cdots, \lambda_n^t)\in \mathcal P$ where $\lambda_j^t=|\{i|\lambda_i\geq j\}|$, called the \textbf{conjugate partition} of $\lambda$.

For a nilpotent matrix $A\in sl(n,\mathbb C)$ satisfying $A^m=0$ but $A^{m-1}\neq 0$, we have a filtration 
\begin{equation}\label{conjugateOrderedPartition}
0=\ker(A^0)\subset \ker(A)\subset \ker(A^2)\subset\cdots\subset \ker (A^m)=V=\mathbb C^n,
\end{equation} and the $k_p=\dim \ker (A^p)-\dim \ker(A^{p-1})  (p>0)$ satisfies $\sum\limits_{k=1}^nk_p=n$ and $k_p\geq k_{p+1}$. 
Then the array $(k_1, k_2, \cdots, k_m)$ forms a partition of $n$, called the \textbf{kernel partition} of $A$.

\begin{lem}\label{Duality}
(i) The conjugate partition of the kernel partition of a nilpotent element $A\in sl(n,\mathbb C)$ is the Jordan type of $A$.\\
(ii) If $\lambda, \mu$ are two partitions of $n$, then $\lambda\geq \mu$ if and only if $\mu^t\geq \lambda^t$.
\end{lem}
The Young diagram helps us see the conjugate partition more explicitly.  For $\lambda\in \mathcal P(n)$, let $k$ be the largest index such that $\lambda_k>0$ and form $k$ rows of empty boxes such that the $i$th row has $\lambda_i$ boxes. Such array is called the \textbf{Young diagram} of $\lambda$. Then we can form a new Young diagram whose rows are the columns of the old ones from left to right, which is the Young diagram of $\lambda^t$. For example, for the partition $\lambda=(6,4,2,1)$ of $13$, we show in the following picture how to obtain the transpose partition $\lambda^t$.
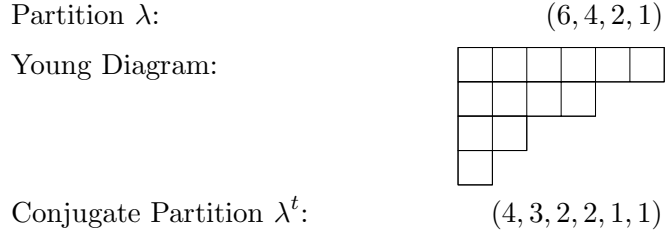
\begin{align*}
&\text{Partition $\lambda$:}&(6,4,2,1)\\
&\text{Young Diagram:}&\yng(6,4,2,1)\\
&\text{Conjugate Partition $\lambda^t$:}&(4,3,2,2,1,1)\end{align*}

\subsection{Compositions, conjugate partitions and an analogue of Young diagram}\label{Composition}
A \textbf{composition} of $n$ is an array $(r_i)$ of positive integers $r_1,\cdots, r_m$ satisfying $\sum\limits_{i=1}^mr_i=n$. Note that $r_i$'s are not necessarily in non-increasing order and compositions are generalizations of partitions. Denote by $\mathcal{C}_n$ the space of compositions of $n$. For a composition $\mathcal R\in \mathcal{C}_n$, we can also define a ``conjugate partition\rq\rq{} of $\mathcal R$.  Define the set $T_i:=\{j|r_j\geq i\}$ and let $S_i$ be the set of lengths of consequent integers inside $T_i$. For example, if $T_i=\{2,3,4,5,7\}$, then there are two sequence of consequent integers $\{2,3,4,5\}$ and $\{7\}$ and hence $S_i=\{4,1\}$, which forms a partition of $T_i$. The collection of the sets $S_i$ are called the \textbf{conjugate set partition} of $\mathcal R$. The disjoint union of partitions $\lambda_1,\cdots, \lambda_t$ form a partition of $n$ and we call it the \textbf{conjugate partition} of $\mathcal R$ and denote it by $\mathcal{R}^t$.  For technical reasons, we extend the definition of conjugate partitions to the compositions which allow to add zeros before and after a composition. 

When $\mathcal R$ is in fact a partition, we can see that its conjugate partition coincides with its conjugate partition. So the conjugate partition is a generalization from partitions to compositions. For $\mathcal R\in \mathcal{C}_n$, the conjugate partition of $\mathcal R^t$ is $\mathcal R$ if and only if $\mathcal R\in \mathcal P_n$. 

A composition $\mathcal R=(r_1,\cdots, r_m)\in \mathcal C_n$ induces a partition $\lambda_{\mathcal R}\in \mathcal P_n$ by reordering $r_{i_1}\geq r_{i_2}\geq\cdots\geq r_{i_m}$.

We can extend the dominance order of $\mathcal P_n$ to the space $\mathcal C_n$. For $\mathcal R_1=(r_1,\cdots, r_m)$ and $\mathcal R_2=(s_1,\cdots, s_l)$. Then we say $\mathcal R_1$ dominates $ \mathcal R_2$ ($\mathcal R_1\geq \mathcal R_2$),  if for each $p$, $\sum\limits_{i=1}^pr_i\geq \sum\limits_{i=1}^p s_i$. 

For each $\mathcal R=(r_1, \cdots, r_m)\in \mathcal{C}_n$, we can associate a constant \begin{equation}\label{ConstantD}D_{\mathcal R}:=\frac{4}{\sum\limits_{p=1}^m p(p-1)r_p}.\end{equation} We can restrict to define $D_{\lambda}$ for $\lambda\in \mathcal P_n$. 

\begin{lem}\label{OrderOfPartition}
(i) For a partition $\lambda\in \mathcal P_n$, we have $C_{\lambda}=D_{\lambda^t}$.\\
(ii)(a) For two compositions $\mathcal R_1, \mathcal R_2\in \mathcal C_n$, if $\mathcal R_1< \mathcal R_2$, we have the constant $D_{\mathcal R_1}<D_{\mathcal R_2}$.\\ (b) For two partitions $\lambda,\mu\in \mathcal{P}_n$, if $\lambda>\mu$, then the constant $C_{\lambda}<C_{\mu}$. \\
(iii) For a composition $\mathcal R\in \mathcal P_n$ with $\lambda_{\mathcal R}$ its induced partition, then their conjugate partitions satisfy $\mathcal R^t\leq (\lambda_{\mathcal R})^t$ and $C_{\mathcal R^t}\geq C_{(\lambda_{\mathcal R})^t}=D_{\lambda_{\mathcal R}}$.\\
(iv) For any permutation of $r_i$\rq{}s in $\mathcal R$ which is still a composition of $n$, denoted as $\mathcal R_1$, then $\mathcal D_{\mathcal R_1}\leq C_{\mathcal R^t}.$
\end{lem}
\begin{proof}
Part (i) is by direct calculation.  Suppose $\lambda=(\lambda_1,\cdots, \lambda_k)$ and $\lambda^t=(\mu_1,\cdots, \mu_l)$. Denote $d_p=p(p^2-1)$ and then $C_{\lambda}=\frac{12}{\sum\limits_{p=1}^k\lambda_p(\lambda_p^2-1)}=\frac{12}{\sum\limits_{p=1}^kd_{\lambda_p}}$. By the duality of $\lambda$ and $\lambda^t$, we have that inside $\lambda$, $i$ appears exactly $\mu_i-\mu_{i+1}$ times. (Assume $\mu_{l+1}=0$). Recall the formula (\ref{ConstantC}) of $C_{\lambda}$, we can rewrite $$C_{\lambda}=\frac{12}{\sum\limits_{q=1}^l(\mu_q-\mu_{q+1})\cdot B_{q}}=\frac{12}{\sum\limits_{q=1}^l \mu_q\cdot (d_q-d_{q-1})}=\frac{4}{\sum\limits_{q=1}^l \mu_p\cdot q(q-1)}=D_{\lambda^t}.$$\\

For Part (ii), denote $A_p=\sum\limits_{i=1}^pr_i$ and hence $r_p=A_p-A_{p-1}$. Then \begin{eqnarray*}
\frac{4}{D_{\mathcal R}}&=&\sum\limits_{i=1}^m r_p\cdot p(p-1)=\sum_{i=1}^m (A_p-A_{p-1})\cdot p(p-1)\\&=&\sum\limits_{i=1}^{m-1} A_p(p(p-1)-(p+1)p)+A_m m(m-1)\\
&=&-2\sum\limits_{i=1}^mA_p\cdot p+nm(m-1).
\end{eqnarray*} Then we can see that $\frac{4}{D_{\mathcal R}}$ is decreasing with respect to each $A_p$. Therefore, we obtain Part (ii)(a). 

For Part (ii)(b), given two partitions $\lambda, \mu\in \mathcal P_n$ satisfying $\lambda>\mu$, it follows immediately from Part (ii) of Lemma \ref{Duality} that $\lambda^t<\mu^t$. So by Part (ii)(a), we have $D_{\lambda^t}<D_{\mu^t}$. Then by Part (i), we obtain $C_{\lambda}<C_{\mu}$. Conversely, it is also true.

For Part (iii), in the definition of $\mathcal R^t$, it is the disjoint union of partitions $\lambda_1,\cdots, \lambda_t$ where each $\lambda_i$ is a partition of $l_i=|T_i|$. It is easy to see that for the induced partition $(\lambda_{\mathcal R})^t$, it is the disjoint union of partitions $(l_1),\cdots, (l_t)$. For each $i$, $\lambda_i\leq (l_i)$ and thus by Lemma \ref{DisjointUnionDominance}, we have $$\mathcal R^t=\lambda_1\cup\lambda_2\cup\cdots \cup\lambda_t\leq (l_1)\cup (l_2)\cup\cdots\cup (l_t)=(\lambda_{\mathcal R})^t.$$ Then we have 

For Part (iv), clearly for any permutation of $r_i$\rq{}s in $\mathcal R$ which is still a composition of $n$ and is always dominated by the induced partition $\lambda_{\mathcal R}$ of $\mathcal R$. By Part (v), we have $D_{\mathcal R_1}\leq D_{\lambda_{\mathcal R}}$. Together with $C_{\mathcal R^t}\geq D_{\lambda_{\mathcal R}}$ in Part (iii), we obtain that $\mathcal D_{\mathcal R_1}\leq C_{\mathcal R^t}.$
\end{proof}

\begin{lem}\label{DUAL}For any $\mathcal{R}\in \mathcal{C}_n$ with only $m$ positive integers, the conjugate partition of $\mathcal{R}^t$ is $(a^1,a^2-a^1,\cdots, a^m-a^{m-1})(\mathcal R)$ where 
\begin{eqnarray}
&&a_l^i(\mathcal R)=r_{l}-\min\limits_{l\leq t\leq l+i}r_t (\text{for $l\leq m-i$)},\\
&&a^i(\mathcal R):=\sum\limits_{l=1}^{m-i}a_l^i(\mathcal R)+\sum\limits_{p=m-i+1}^{m}r_{p} (a^m(\mathcal R)=n).\end{eqnarray} 
\end{lem}
\begin{proof}
We prove by induction on the number of positive numbers in a composition. If it is $1$, the statement is clearly true. Assume that the statement holds for any $k<m$ in which case the composition has $k$ positive numbers.
Suppose there are exactly $i_1,\cdots,i_k$ satisfying $1\leq i_1<i_2<\cdots<i_k\leq m$ and $r_{i_p}=1, 1\leq p\leq k.$ Set $r_l'=r_l-1, 1\leq l\leq m.$ Set $i_0:=0$ and $i_{k+1}:=m+1$. For each $0\leq p\leq k$, let $\mathcal R^p$ be $(0,\cdots, 0, r_{i_p+1}',\cdots, r_{i_{p+1}-1}',0,\cdots,0)$. Note that if $i_1=1$, $\mathcal R^0$ is a zero vector; if $i_k=m$, $\mathcal R^{k}$ is a zero vector.  Then for $l\leq m-i$,
\begin{eqnarray*}
&&a_l^i(\mathcal R^p)=r^p_{l}-\min\limits_{l\leq t\leq l+i}r^p_t=0\quad\text{for $l\leq i_p$ or $l\geq i_{p+1}$,}\\
&&a_l^i(\mathcal R^p)=r^p_{l}-\min\limits_{l\leq t\leq l+i}r^p_t=r_{l}'-\min\limits_{l\leq t\leq l+i}r_{t}'=a_l^i(\mathcal R)\quad \text{ for $i_p\leq l\leq i_{p+1}-1$}.
\end{eqnarray*}

For a fixed $i$ satisfying $i_p\leq m-i\leq i_{p+1}-1$ for some $p$, then 
\begin{eqnarray*}
 a^i(\mathcal R^f)&=&\sum\limits_{l=1}^{m-i}a_l^i(\mathcal R^f)+\sum\limits_{q=m+1-i}^{m}\mathcal R^f_{q}=\sum\limits_{l=i_f}^{i_{f+1}-1}a_l^i(\mathcal R)\quad \text{for $0\leq f<p$,}\\
a^i(\mathcal R^f)&=&\sum\limits_{l=1}^{m-i}a_l^i(\mathcal R^f)+\sum\limits_{q=m+1-i}^{m}\mathcal R^f_{q}=\sum\limits_{l=i_f}^{i_{f+1}-1}r'_l=\sum\limits_{l=i_f}^{i_{f+1}-1}(r_l-1)\text{for $p<f\leq k$,}\\
a^i(\mathcal R^p)&=&\sum\limits_{l=1}^{m-i}a_l^i(\mathcal R^p)+\sum\limits_{q=m+1-i}^{m}\mathcal R^p_{q}=\sum\limits_{l=i_p}^{m-i}a_l^i(\mathcal R)+\sum\limits_{q=m+1-i}^{i_{p+1}-1}r'_q\\&=&\sum\limits_{l=i_p}^{m-i}a_l^i(\mathcal R)+\sum\limits_{q=m+1-i}^{i_{p+1}-1}(r_q-1).\end{eqnarray*}
Summing over the above equalities,  
we obtain 
\begin{eqnarray*}
a^i(\mathcal R)&=&\sum\limits_{l=1}^{m-i}a_l^i(\mathcal R)+\sum\limits_{p=m+1-i}^mr_p\\
&=&\sum\limits_{p=0}^{k}\sum\limits_{l=i_p+1}^{i_{p+1}-1}a_l^i(\mathcal R^p)+\sum\limits_{p=m+1-i}^mr'_p+i=\sum\limits_{p=0}^{k}a^i(\mathcal R^p)+i.\end{eqnarray*}
So we obtain the first fact: $a^i(\mathcal R)-a^{i-1}(\mathcal R)$ has $1$ more than $\sum\limits_{p=0}^{k}(a^i(\mathcal R^p)-a^{i-1}(\mathcal R^{p}))$ for every $i$. 

By assumption, for each $\mathcal R^p$, $a^i(\mathcal R^p)-a^{i-1}(\mathcal R^p)$ counts the number of elements in $(\mathcal R^p)^t$ which are at least $i$. By the definition of the conjugate partition, $(\mathcal{R})^t$ is a disjoint union of the set $\{m\}$ and $(\mathcal R^0)^t, (\mathcal R^1)^t, \cdots, (\mathcal R^{k})^t$. So we have the second fact: for each $i$, the number of elements in $(\mathcal{R})^t$ at least $i$ is exactly $1$ more than the number of the elements in the disjoint union of $(\mathcal R^0)^t, (\mathcal R^1)^t, \cdots, (\mathcal R^{k})^t$ at least $i$. 

Combining the two facts, $a^i(\mathcal R)-a^{i-1}(\mathcal R)$ is exactly the number of elements in $\mathcal{R}^t$ which are at least $i$ and the statement is proven for $m$.
\end{proof}

\begin{df}\label{TypeR}Given a composition $\mathcal{R}\in \mathcal {C}_n$, a matrix $A$ is said to be of \textbf{type $\mathcal R$} if there exists a decomposition of $V=V_1\oplus \cdots\oplus V_m$ where $V_i$ is of rank $r_i$, such that the linear transformation of $A$ on $V$ is just the combinations of each $A_i: V_i\rightarrow V_{i+1}$. That is, with respect to the decomposition, $A$ is of the form \begin{equation}\begin{pmatrix}
0&&&&\\
A_1&0&&&\\
&A_2&0&&\\
&&\ddots&\ddots&\\
&&&A_{m-1}&0\end{pmatrix},\end{equation}where each $A_i$ is a $r_{i}\times r_{i+1}$ matrix, for $1\leq i\leq m-1$. 
\end{df}

\begin{lem}\label{TypeMatrix} For a matrix $A$ of type $\mathcal R\in \mathcal{C}_n$, then\\
(i) $A$ is of Jordan type at most $\mathcal{R}^t$;\\
(ii) $A$ is of Jordan type $\mathcal{R}^t$ if and only if each $V_i$ decomposes into subspaces $V_i^1, V_i^2,\cdots, V_i^{s_i}$ where $A_i|_{V_i^s}$ is either an isomorphism to $V_{i+1}^s$ or a zero map if $V_{i+1}^s$ does not exist. 
\end{lem}
Part (ii) is not needed for our main result. We believe it is of its own interest and include the proof here. 
\begin{proof}

For Part (i), it is equivalent to show that for each $i$, $\dim ker A^i\geq a^i(\mathcal R)$. Note that the composition $A_{l+i-1}\circ A_{l+i-2}\circ\cdots\circ A_l:V_l\rightarrow V_{l+i}$ has rank no greater than that of $A_l,\cdots,A_{l+i}$. So $$\dim ker (A_{l+i-1}\circ A_{l+i-2}\circ\cdots \circ A_l)\geq r_l-\min_{l\leq p\leq l+i}r_p=a^i_l(\mathcal R).$$ Observe that the subspace $ker A^i$ is a direct sum of $\bigoplus\limits_{l=1}^{m-i}ker(A_{l+i-1}\circ A_{l+i-2}\circ\cdots \circ A_l)$ and $\bigoplus\limits_{p=1}^iV_{m-p+1}$. So $\dim ker A^i\geq \sum\limits_{l=1}^{m-i}a_l^i(\mathcal R)+\sum\limits_{p=1}^ir_{m-p+1}=a^i(\mathcal R)$.\\

For Part (ii), for the equality holds, we have the following claim:
\begin{cla}\label{ImportantClaim}
For each $l, i$ satisfying $1\leq l, l+i\leq m$, the map $A_{l+i-1}\circ A_{l+i-2}\circ\cdots\circ A_l:V_l\rightarrow V_{l+i}$ has rank equal to $\min\limits_{l\leq p\leq l+i}r_p$.
\end{cla} 

We prove by induction on $n$. In the case $n=1$, the statement is clear. Suppose the statement is true for all $k<n$ under the assumption \ref{ImportantClaim}. Now we are going to show it is true for the case $n$. Suppose there are exactly $s_1<s_2<\cdots<s_k$ such that $r_{s_1}=\cdots =r_{s_k}=b=\min_{1\leq i\leq m}r_i$. Set $s_0=0$ and $s_{k+1}=n+1.$ For each $1\leq p\leq k$, define the subspace $V_{s_p}^1=V_{s_p}$. For $i\neq s_1,\cdots, s_{k}$, let $s_p$ be the nearest number of $i$ (might not be unique), then define
\begin{eqnarray}
V_i^1=(A_{s_p-1}\circ A_{s_p-2}\circ\cdots\circ A_i)^{-1}(V_{s_p}),\quad \text{for $i< s_p$};\\
V_i^1=(A_{i-1}\circ A_{i-2}\circ\cdots\circ A_{s_p})(V_{s_p}),\quad \text{for $i>s_p$}.
\end{eqnarray}We only need to check whether $V_i^1$ is well-defined. Suppose there exist $s_p, s_{p+1}$ such that $i=\frac{s_p+s_{p+1}}{2}$. From Claim \ref{ImportantClaim} and $r_{s_p}=r_{s_{p+1}}=b$, $A_{s_{p+1}-1}\circ A_{s_{p+1}-2}\circ\cdots\circ A_{i_p}: V_{s_p}\rightarrow V_{s_{p+1}}$ is an isomorphism, and thus $$(A_{s_{p+1}-1}\circ A_{s_{p+1}-2}\circ\cdots\circ A_i)^{-1}(V_{s_{p+1}})=(A_{i-1}\circ A_{i-2}\circ \cdots\circ A_{s_p})(V_{s_p}).$$ Therefore the two definitions of $V_i^1$ coincide and hence $V_i^1$ is well-defined. Moreover, $A|_{V_i^1}$ is always an isomorphism to $V_{i+1}^1$.

For each $0\leq p\leq k$, take $W_{s_p+1}$ to be the complement of $V_{s_p+1}^1$ in $V_{s_p+1}$. For each $i$ satisfying $s_p<i<s_{p+1}$, let $W_i$ be the complement of $V_i^1$ in $V_i$ containing $A^{i-s_p-1}(W_{s_p+1})$ and $V_i$ is decomposed into $V_{i}^1+W_i$.
Set \begin{eqnarray}\label{DecompositionUp}
&&U_p:=W_{s_p+1}\oplus W_{s_p+2}\oplus\cdots \oplus W_{s_{p+1}-1}\quad\text{ for $0\leq p\leq k.$}\end{eqnarray}
We can decompose the space into the direct sum of $A$-invariant subspaces $\bigoplus\limits_{i=1}^m V_i^1$ and $U_p$'s for $0\leq p\leq k$. On each subspace $U_p$, with respect to the decomposition (\ref{DecompositionUp}), $A$ is of the form 
\begin{equation}\begin{pmatrix}
0&&&&\\
A_1^p&0&&&\\
&A_2^p&0&&\\
&&\ddots&\ddots&\\
&&&A_{s_{p+1}-s_p-2}^p&0\end{pmatrix},\end{equation}
Moreover, following from Claim \ref{ImportantClaim}, we still have that for each $l, i$ such that $s_p+1\leq l, l+i\leq s_{p+1}-1$, the map $A_{l+i-1}^p\circ A_{l+i-2}^p\circ\cdots\circ A_l^p:W_l\rightarrow W_{l+i}$ has rank equal to $\min\limits_{l\leq q\leq l+i}(r_q-b)$. On each $U_p$, it satisfies Claim \ref{ImportantClaim}. By assumption, we have the decomposition on each $U_p$. Together with the space $\bigoplus\limits_{i=1}^m V_i^1$, we obtain the statement. 
\end{proof}

To see the conjugate partition clearly, we introduce a diagram. For $\mathcal R=(r_1,\cdots,r_m)\in \mathcal C(n)$, form $m$ rows of empty boxes such that the $i$th row has $r_i$ boxes. We obtain an analogue of Young diagram. Unlike the Young diagram which only has the increasing rows, here we consider rows which are not necessarily non-increasing. We first define a \textbf{conjugate set partition} of $\mathcal R$ as $(\lambda_1,\lambda_2,\cdots,\lambda_k)$, where $\lambda_i$ is a partition of $T_i$. For each column, we read the length of each connected component of empty boxes and denote by $S_i$ the set of lengths in the $i_{th}$ column. The conjugate partition $\mathcal R^t$ of $\mathcal R$ is the disjoint union of the conjugate set partition. For example, for a composition $\mathcal R$ of $13$, we show in the following picture how to obtain the conjugate partition $\mathcal R^t$.
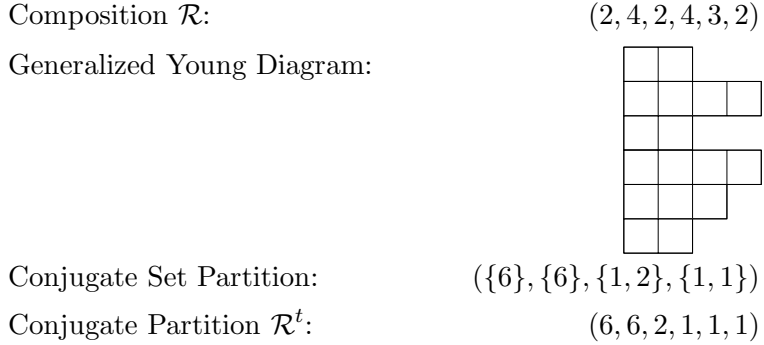
\begin{align*}
&\text{Composition $\mathcal R$:}&(2,4,2,4,3,2)\\
&\text{Generalized Young Diagram:}&\yng(2,4,2,4,3,2)\\
&\text{Conjugate Set Partition:}&(\{6\}, \{6\}, \{1, 2\}, \{1, 1\})\\
&\text{Conjugate Partition $\mathcal R^t$:}&(6,6,2,1,1,1)
\end{align*}

\begin{prop}\label{CompositionEstimate}
For a matrix $A$ of type $\mathcal R\in \mathcal{C}_n$ together with the decomposition $V=V_1\oplus \cdots\oplus V_m$, the function satisfies \begin{equation}
K(A)\geq C_{\mathcal R^t}.\end{equation}

In the case when $V_i$'s are orthogonal, each $V_i$ decomposes into $r_i$ lines $L_i^1, L_i^2,\cdots, L_i^{r_i}$ orthogonally, where $A_i|_{L_i^s}$ is either an isometry to $L_{i+1}^s$ or a zero map if $L_{i+1}^s$ does not exist. Suppose for each $L_i^s$, let $I_i^s=[j_0^s, j_1^s]$ be the maximal interval containing $i$ such that $A_j|_{L_j^s}$ is nonzero for each $j\in \mathbb Z\cap I_i^s$ and there is a unit vector $v_j^s$ of $L_j^s$, $A_j(v_j^s)=a_j^sv_{j+1}^s$ where $a_j^s=\sqrt{(j-j_0^s+1)(j_1^s-j)}$. Then $K(A)=C_{\mathcal R^t}.$
\end{prop}
\begin{proof}
Following from Lemma \ref{JordanInequality} and \ref{TypeMatrix}, we obtain the inequality. 

The equality holds since the vector space decomposes into irreducible $A$-invariant orthogonal subspaces such as $L_i^s$ for $j_0\leq i\leq j_1$. Moreover, in each such subspace of dimension $m$, $A$ is unitarily conjugate to $J_m$ as in assumption. So the equality $K(A)=C_\lambda$ follows from the equality condition in Proposition \ref{JordanInequality}.
\end{proof}

\begin{rem}\label{Symmetry}
If $\mathcal R=(r_1,\cdots, r_m)$ is symmetric, that is, $r_i=r_{m+1-i}$. Suppose for $L_i^s$, let $I_i^s$ be the maximal interval containing $i$ such that $A_j|_{L_j^s}=[j_0^s, j_1^s]$ is nonzero for each $j\in \mathbb Z\cap I_i^s$. Then by the symmetry, for $L_{m+1-i}^s$, the maximal interval $I_{m+1-i}^s$ containing $m+1-i$ such that $A_j|_{L_j^s}$ is nonzero for each $j\in \mathbb Z\cap I_{m+1-i}^s$ is $[m+1-j_1^s, m+1-j_0^s]$ and thus $a_j^s=a_{m-j}^s.$
\end{rem}

\section{Nilpotent Higgs bundles on general complex manifolds}\label{GeneralComplexManifolds}
\begin{df}
A Higgs bundle over a complex manifold $M$ consists of a pair $(E,\phi)$ where $E$ is a holomorphic vector bundle over $M$ and $\phi$ is a holomorphic section of $\Omega^1(End(E))$ satisfying the integrability condition $\phi\wedge \phi=0$.
\end{df}
A Hermitian metric $h$ on a degree $0$ Higgs bundle $(E,\phi)$ is called \textit{harmonic} if it satisfies the Hitchin equation 
\[F(D^h)+[\phi,\phi^{*_h}]=0,\quad \partial^h\phi=0,\]
where $D^h$ is the Chern connection on $E$ uniquely determined by the holomorphic structure and the metric $h$, $F(D^h)$ is the curvature of $D^h$, $\phi^{*_h}$ is the Hermitian adjoint of $\phi$ with respect to $h$ and $\partial^h$ is the $(1,0)$-part of $D^h$. On $End(E)=E\otimes E^*$, there is an induced metric: for $X, Y\in End(E),  (X, Y)=tr(XY^{*_h})$. The harmonic metric $h$ gives a Hermitian metric $g_M$ on $M$:
\[g_M(\frac{\partial}{\partial z_j}, \frac{\partial}{\partial z_k})=(\phi_j, \phi_k)=tr(\phi_j\phi_k^{*_h}),\]
where $\phi_j=\phi(\frac{\partial}{\partial z_j})$. The metric $g_M$ is called \textit{Hodge metric} on $M$, which is in fact K\"ahler (see also Wang \cite{XuWang}):\\
The fundamental form of $g_M$ is $\omega=\frac{i}{2} \cdot tr(\phi_j\phi_k^{*_h})dz_j\wedge d\bar z_k$, then \begin{eqnarray*}\partial \omega&=&[tr((\partial \phi_j)\phi_k^{*_h})-tr(\phi_j[h^{-1}\partial h, \phi_k^{*_h}])]dz_j\wedge d\bar z_k\\
&=&\frac{i}{2}[tr((\partial \phi_j)\phi_k^{*_h})+tr([h^{-1}\partial h, \phi_j] \phi_k^{*_h})]dz_j\wedge d\bar z_k\\
&=&\frac{i}{2} \cdot tr((\partial^h\phi_j)\phi_k^{*_h})dz_j\wedge d\bar z_k=0,
\end{eqnarray*} where the last equality follows from the Hitchin equation. 
Also, $\bar\partial \omega=\frac{i}{2}\cdot tr(\phi_j(\bar\partial\phi_k^{*_h}))dz_j\wedge d\bar z_k=\frac{i}{2}\cdot  tr(\phi_j(\partial^h\phi_k)^{*_h})dz_j\wedge d\bar z_k=0.$ So $\omega$ is closed and hence $g_M$ is K\"ahler.

\begin{prop}\label{Reducing1}
Let $(E,\phi)$ be a degree $0$ Higgs bundle over a complex manifold $M$ which admits a harmonic metric $h$. Then we have away from zeros of $\phi_j$, the holomorphic sectional curvature $\kappa_j$ of the Hodge metric $g_M$ on the tangent plane $span_{\mathbb C}\{\frac{\partial}{\partial{z_j}}\}$ is 
\begin{equation}
\kappa_j\leq -\frac{||[\phi_j,\phi_j^{*h}]||^2}{||\phi_j||^4},
\end{equation} where $(X,Y)=tr(XY^{*_h})$ and $||X||^2=(X,X).$
\end{prop}
\begin{proof}
For a holomorphic vector bundle $E$ together with a Hermitian metric $h$, we have the associated Chern connection, denoted $D^h$. The curvature operator of the Chern connection is denoted $\phi^V$, locally it is $\sum\limits_{j,k}\phi_{j\bar k}^V dz^j\wedge d\bar z^k$. We can see the $T_M$ as a subbundle of $End(E)$ away from zeros of $\phi$. By the curvature formula for the subbundle, for $v, w\in TM$,
\begin{eqnarray*}
(\Omega_{j\bar k}^{T_M}v,w)_E&=&(\Omega_{j\bar k}^{End(E,h)}\phi(v),\phi(w))\\&&-(P^{\perp}(D_j^{End(E,h)}(\phi(v))),P^{\perp}(D_k^{End(E,h)}(\phi(w)))),
\end{eqnarray*} where $\Omega_{j\bar k}^{End(E,h)}$ are the curvature operators of the Chern connections and $P^{\perp}$ denotes the orthogonal projection to the orthogonal complement of $T_M$ in $End(E,h)$. The curvature formula on $End(E,h)$ is  
\begin{equation}
\Omega_{j\bar k}^{End(E,h)}X=[\Omega_{j\bar k}^h,X],\quad X\in End(E)
\end{equation} where $\Omega^h$ denotes the curvature of the Chern connection $D^h$ on $(E,h)$.
The harmonic metric $h$ satisfies the Hitchin equation
\begin{equation}
F(D^h)+[\phi,\phi^{*h}]=0.
\end{equation}
Equivalently, \begin{equation}
\Omega_{j\bar k}^h+[\phi_j,\phi_k^{*h}]=0.
\end{equation} 
So for any $v, w\in TM$, \begin{eqnarray*}
&&(\Omega_{j\bar k}^{End(E,h)}\phi(v),\phi(w))\\&=&([\Omega_{j\bar k}^h,\phi(v)],\phi(w))=([[\phi_k^{*h},\phi_j],\phi(v)],\phi(w))\\
&&\text{using the Jacobi identity: $[A,[B,C]]+[B,[C,A]]+[C,[A,B]]=0$}\\
&=&([[\phi_k^{*h},\phi(v)],\phi_j],\phi(w))-([\phi_k^{*h},[\phi(v),\phi_j]],\phi(w))\\
&&\text{the integrability condition $\phi\wedge\phi=0$ implies that $[\phi(v),\phi_j]=0$}\\
&=&([[\phi_k^{*h},\phi(v)],\phi_j],\phi(w))\\
&=&([\phi_k^{*h},\phi(v)]\phi_j-\phi_j[\phi_k^{*h},\phi_j],\phi(w))=([\phi_k^{*h},\phi(v)],\phi(w)\phi_j^{*h}-\phi_j^{*h}\phi(w))\\
&=&-([\phi_k^{*h},\phi(v)],[\phi_j^{*h},\phi(w)]).
\end{eqnarray*}
So the holomorphic sectional curvature is 
\begin{eqnarray*}
\kappa_j&=&\frac{(\phi_{j\bar j}^{T_M}\phi_j,\phi_j)}{(\frac{\partial}{\partial z_j},\frac{\partial}{\partial z_j})(\frac{\partial}{\partial\bar z_j},\frac{\partial}{\partial\bar z_j})}\leq\frac{(\phi_{j\bar j}^{End(E)}\phi_j,\phi_j)}{(\frac{\partial}{\partial z_j},\frac{\partial}{\partial z_j})(\frac{\partial}{\partial\bar z_j},\frac{\partial}{\partial\bar z_j})}\\
&=&-\frac{([\phi_j^{*h},\phi_j],[\phi_j^{*h},\phi_j])}{(\phi_j,\phi_j)(\phi_j^{*_h},\phi_j^{*_h})}=-\frac{||[\phi_j,\phi_j^{*h}]||^2}{tr(\phi_j\phi_j^{*_h})^2} .
\end{eqnarray*}
\end{proof}

So we obtain the following proposition.
\begin{prop}\label{NilpotentCurvature}
Suppose at any point $p$ such that $\phi_j$ is nilpotent of Jordan type at most $\lambda\in\mathcal{P}_n$, then the holomorphic sectional curvature $k_j(p)$ of the Hodge metric over $M$ is bounded from above by $-C_{\lambda}.$
\end{prop}
\begin{proof}
Let $A$ be the matrix representation of $\phi_j(p)$ under a unitary basis of $h$, then by Proposition \ref{Reducing1}, the holomorphic sectional curvature satisfies
$k_j(p)\leq -K(A)$ where the function $K$ is defined in Equation (\ref{DefinitionK}). The matrix $A$ is nilpotent and the Jordan type of $A$ coincides with the Jordan type of $\phi_j$, it follows form Proposition \ref{JordanInequality} that $K(A)\geq C_\lambda$ and thus $k_j(p)\leq -C_{\lambda}$. 
\end{proof}

We then prove the following corollary.
\begin{cor} Suppose the Higgs field $\phi$ is $k$-nilpotent (i.e. $\phi^{k}=0$) and admits a harmonic metric. Then the holomorphic sectional curvature of the Hodge metric over $M$ is bounded above by $$-\frac{12}{n(k^2-1)-s(k^2-s^2)}\leq-\frac{12}{n(k^2-1)},$$ where $n=kd_0+s$ where $s\leq k-1$.
\end{cor}
\begin{proof}
Note that among all Jordan types of $k$-nilpotent matrices, the partition $(k,k,\cdots,k,s) (s\leq k-1)$ is the maximum.  Using Proposition \ref{NilpotentCurvature}, we obtain the theorem.
\end{proof}

Among nilpotent Higgs bundles of rank $n$, there is a special family called \textit{complex variation of Hodge structures} ($\mathbb C$-VHS). They are of the form $$(E=E_1\oplus E_2\oplus\cdots\oplus E_k,\quad\phi=\begin{pmatrix}0&&&&\\\alpha_1&0&&&\\&\alpha_2&0&&\\&&\ddots&\ddots&\\&&&\alpha_{k-1}&0\end{pmatrix}),$$ where $E_k$ is a holomorphic vector bundle of rank $r_k$ and $\alpha_j$ is a holomorphic map$:E_j\rightarrow \Omega^1(E_{j+1})$. The  ordered partition of $n$, $(r_1,\cdots, r_k)$ is a composition of $n$ and denote by $\mathcal C_n$ the space of compositions of $n$. We call the above $(E,\phi)$ a $\mathbb C$-VHS of type $\mathcal R\in \mathcal C_n$. 

Following from Lemma \ref{TypeMatrix} and Proposition \ref{NilpotentCurvature}, we have
\begin{cor}\label{CVHSCurvature}
Suppose the complex variation of Hodge structure $(E,\phi)$ over a complex manifold $M$ is of type $\mathcal R\in \mathcal C_n$ and admits a harmonic metric. Suppose $\mathcal R^t$ is the conjugate partition of $\mathcal R$. Then the holomorphic sectional curvatures of the Hodge metric on $M$ are bounded from above by $-C_{\mathcal R^t}.$ 
\end{cor}

\section{Nilpotent Higgs bundles on Riemann surfaces}\label{RiemannSurface}
In this section, let $\Sigma$ be a closed Riemann surface of genus $g\geq 2$. We will obtain topological information of nilpotent polystable Higgs bundles over $\Sigma$ in Theorem \ref{NilpotentDomination}. 

Let us recall the nonabelian Hodge correspondence. An $SL(n,\mathbb C)$-Higgs bundle over $\Sigma$ is a Higgs bundle $(E,\phi)$ satisfying $\det E=\mathcal O$ and $tr\phi=0$. It is \textit{stable} if any $\phi$-invariant holomorphic subbundle $F$ of $E$ has negative degree and \textit{polystable} if it is a direct sum of stable Higgs bundles of degreee $0$. On $X:=SL(n,\mathbb C)/SU(n)$, the form $(A,B)=2tr(AB)$ on $sl(n,\mathbb C)$ induces a $SL(n,\mathbb C)$-invariant Riemannian metric. We choose such a normalization so that $SL(2,\mathbb C)/SU(2)$ is of constant curvature $-1$.  For a polystable $SL(n,\mathbb C)$-Higgs bundle over $\Sigma$, there exists a harmonic metric $h$ and the connection $\nabla_{\bar\partial_E,h}+\phi+\phi^{*_h}$ is flat. Let $\rho:\pi_1(\Sigma)\rightarrow SL(n,\mathbb C)$ be the holonomy of $\nabla$. By parallel transporting the harmonic metric with respect to the flat connection $\nabla$, we obtain a $\rho$-equivariant harmonic map $f: \widetilde\Sigma\rightarrow X. $ There is a close relationship between the harmonic maps and the Higgs bundles, see \cite{Li} Section 5:
\begin{eqnarray}
&&f^*g_X=2tr(\phi^2)dz^2+2tr(\phi\phi^{*_h})(dz\otimes d\bar z+d\bar z\otimes dz)+2tr((\phi^{*_h})^2)d\bar z^2,\label{PullbackMetric}\\ &&\text{Hopf}(f)=2tr(\phi^2)dz^2.
\end{eqnarray}

Moreover, for the curvature formula, we have a stronger version of Proposition \ref{Reducing1} as follows. 
\begin{prop}\label{Reducing2}(\cite{Li} Section 5)
At an immersed point $p$ of $f$, the tangent plane $\sigma$ at $p$ of $f(\widetilde\Sigma)$, the sectional curvature $k(\sigma)^X$ and the curvature $\kappa(p)$ of the pullback metric $f^*g_X$ satisfy
\begin{equation}
\kappa(p)\leq k(\sigma)^X=-\frac{||[\ti\phi,\ti\phi^{*_h}]||^2}{||\ti\phi||^4-|tr(\ti\phi^2)|^2}(p).
\end{equation}
where locally $\phi=\ti\phi dz$, $(X,Y)=2tr(XY^{*_h})$ and $||X||^2=(X,X).$

Moreover, $\kappa(p)=k(\sigma)^X$ if and only if $p$ is totally geodesic. 
\end{prop}

Given a partition $\lambda=(\lambda_1^{k_1},\cdots, \lambda_r^{k_r})\in \mathcal{P}_n$, define $$\tau_{\lambda}:SL(2,\mathbb{R})\xrightarrow{(\overbrace{\tau_{\lambda_1},\cdots, \tau_{\lambda_1}}^{\text{$k_1$-times}},\cdots\cdots, \overbrace{\tau_{\lambda_r},\cdots,\tau_{\lambda_r}}^{\text{$k_r$-times}})} \prod\limits_{i=1}^{r}SL(\lambda_i,\mathbb{R})^{k_i}\hookrightarrow SL(n,\mathbb{C}).$$ Define the subgroup $\mathfrak{G}_{\lambda}$ of $SU(n)$ as \begin{equation}\label{Centralizer}
\mathfrak{G}_\lambda=\{A\in \text{diag}(U(k_1)\otimes I_{\lambda_1},\cdots,U(k_r)\otimes I_{\lambda_r})\cap SU(n)\},\end{equation}
which lies in the centralizer of $\tau_{\lambda}$ inside $SL(n,\mathbb C)$. Given two representations $j:\pi_1(\Sigma)\rightarrow SL(2,\mathbb{R})$ and $ \mu_{\lambda}:\pi_1(\Sigma)\rightarrow \mathfrak{G}_{\lambda}$, there is a natural well-defined representation $(\tau_{\lambda}\circ j)\cdot \mu_{\lambda}:\pi_1(\Sigma)\rightarrow SL(n,\mathbb{C})$, $\gamma\mapsto(\tau_{\lambda}\circ j)(\gamma)\cdot \mu_{\lambda}(\gamma),$ where the multiplication is the matrix multiplication. 

For a Fuchsian representation $j:\pi_1(\Sigma)\rightarrow SL(2,\mathbb R)$, denote by $f_j:\ti\Sigma\rightarrow \mathbb H^2$ the $j$-equivariant harmonic map, which is in fact a diffeomorphism. Denote $\bar{\tau}_{\lambda}: \mathbb H^2\rightarrow X$ as the induced map from $\tau_{\lambda}$, which is injective. From Theorem 7.2 in \cite{Helgason}, $\bar{\tau}_\lambda$ is a totally geodesic map. Then $f_{\pi, j}=\bar{\tau}_{\lambda}\circ f_j$ is a harmonic map which is equivariant with respect to the representation $(\tau_{\lambda}\circ j)\cdot \mu_{\lambda}$ for any representation $\mu_{\lambda}: \pi_1(\Sigma)\rightarrow \mathfrak{G}_{\lambda}$ and it is a totally geodesic embedding. 

For nilpotent polystable Higgs bundles, we consider the associated $\rho$-equivariant harmonic map $f: \widetilde\Sigma\rightarrow X$ which is in fact conformal since $\text{Hopf}(f)=tr(\phi^2)=0$. So we obtain $\rho$-equivariant minimal immersion of $\widetilde\Sigma$ into $X$ and the pullback metric is $f^*g_X=2tr(\phi\phi^{*_h})(dz\otimes d\bar z+d\bar z\otimes dz)$.  
\begin{thm}\label{NilpotentDomination}
For a nilpotent polystable $SL(n,\mathbb{C})$-Higgs bundle $(E,\phi)$ of Jordan type at most $\lambda=(\lambda_1^{k_1},\cdots,\lambda_r^{k_r})\in\mathcal P_n$ on $\Sigma$, let $\rho$ be the associated representation into $SL(n, \mathbb C)$. Then the induced metric of the associated $\rho$-equivariant minimal surface $f(\widetilde\Sigma)$ in $X$ is \\
(1) strictly dominated by $\frac{2}{C_{\lambda}}\cdot h_\Sigma$, in which case the translation length spectrum satisfies $l_\rho\leq \alpha\cdot l_{\tau_{\pi}\circ j_\Sigma}$ for some positive constant $\alpha<1$;\\
(2) $\frac{2}{C_{\lambda}}\cdot h_\Sigma$, in which case, $\mathbb P(\rho)=\mathbb P(\tau_{\lambda}\circ j_\Sigma)$ and the Higgs bundle $(E,\phi)$ is a direct sum of Higgs bundles $(E_i, \phi_i)(1\leq i\leq r)$ where each $(E_i, \phi_i)$ is the tensor product of $Sym^{\lambda_i-1}(K^{\frac{1}{2}}\oplus K^{-\frac{1}{2}}, \begin{pmatrix}0&0\\1&0\end{pmatrix})$ and $(V_i,0)$ for $V_i$ a polystable holomorphic vector bundle of rank $k_i$ satisfying $\prod_{i=1}^r\det(V_i)^{\lambda_i}=\mathcal O$. 
\end{thm}
 
\begin{proof} 
Let $A$ be the matrix representation of $\phi(z)$ under a unitary frame of $h$. By Proposition \ref{Reducing2} and using $tr(\phi^2)=0$, the curvature $\kappa$ of the pullback metric $f^*g_X$ on $\Sigma$ satsifies $\kappa\leq -\frac{1}{2}K(A)$ where $K$ is the function defined in Equation (\ref{DefinitionK}). The matrix $A$ is nilpotent and its Jordan type coincides with $\phi$, by Proposition \ref{JordanInequality}, $K(A)\geq C_\lambda$ and thus $\kappa\leq -\frac{C_\lambda}{2}.$ 

Let $h_\Sigma$ be the unique conformal hyperbolic metric on $\Sigma$. By the strong maximum principle or Ahlfors lemma, we obtain either $f^*g_X\leq \alpha\frac{2}{C_{\lambda}} \cdot h_\Sigma$ for some positive constant $\alpha<1$, or $f^*g_X=\frac{2}{C_{\lambda}}\cdot h_\Sigma$. In the latter case, the map is totally geodesic and the curvature of the image is the constant $-\frac{C_{\lambda}}{2}$ following from Proposition \ref{Reducing2}.

\begin{lem}(Dai-Li \cite{DaiLiHitchinFiber}, Lemma 4.6 and 4.7)\label{LemmaA}
Let $j$ be a Fuchsian representation and $f_j$ be the associated harmonic map. \\
(1) For a representation $\rho: \pi_1(\Sigma)\rightarrow SL(n,\mathbb C)$, suppose the associated $\rho$-equivariant harmonic map $f_{\rho}:\widetilde \Sigma\rightarrow X$ satisfies $f_{\rho}^*g_X\leq \frac{1}{c}f_j^*g_{\mathbb{H}^2}$, for some $c>0$, then $l_{\rho}\leq \frac{1}{\sqrt{c}}l_{j}$.\\
(2) $\kappa^X_{\bar\tau_{\lambda}\circ f_j}\equiv-c$ and $l_{\tau_{\lambda}\circ j}=\frac{1}{\sqrt{c}}l_j,$ where $c=\frac{C_{\lambda}}{2}$.
\end{lem}
Applying Lemma \ref{LemmaA} to the uniformization representation $j_\Sigma$ of $\Sigma$, we obtain that if the pullback metric satisfies that $f^*g_X\leq \alpha\frac{2}{C_{\lambda}} \cdot h_\Sigma$ for some positive constant $\alpha<1$, then $l_{\rho}\leq \alpha\cdot l_{\tau_{\lambda}\circ j_\Sigma}$.

The rest is to show if the pullback metric coincides with $\frac{2}{C_{\lambda}}\cdot h_\Sigma$, we obtain the rigidity result. For this, we use the following lemma.
\begin{lem} (Dai-Li \cite{DaiLiHitchinFiber}, Section 5) \label{LemmaB}Suppose a $\rho$-equivariant harmonic map $f:\tilde{\Sigma}\rightarrow SL(n,\mathbb{C})/SU(n)$ is a totally geodesic immersion and the curvature of its image is a negative constant $-c$. Then there is a Fuchsian representation $j$ and a representation $\mu_{\pi}:\pi_1(\Sigma)\rightarrow \mathfrak{G}_{\pi}$ such that the corresponding $j$-equivariant harmonic map $f_j:\widetilde{\Sigma}\rightarrow \mathbb H^2$, a partition $\pi\in\mathcal{P}_n$ and an element $x\in SL(n,\mathbb{C})$ such that
$$c=\frac{C_{\pi}}{2},\quad f=L_x\circ\bar{\tau}_{\pi}\circ f_j,\quad\rho=\text{Ad}_{x^{-1}}\circ\big((\tau_{\pi}\circ j)\otimes \mu_{\pi}\big),$$
where $L_x$ is the left action on $SL(n,\mathbb{C})/SU(n)$ and $\text{Ad}$ is the adjoint action on $SL(n,\mathbb{C})$.
\end{lem}

Apply Lemma \ref{LemmaB} to the minimal immersion $f$, we obtain that $f_j$ is also conformal. Therefore, $j$ is the uniformizing representation $j_{\Sigma}$ of $\Sigma$ and the corresponding Higgs bundle is of the form $(K^{\frac{1}{2}}\oplus K^{-\frac{1}{2}}, \begin{pmatrix}0&0\\1&0\end{pmatrix})$.  And $\rho$ is conjugate to $(\tau_{\pi}\circ j)\cdot\mu_{\pi}$, for some $\pi\in \mathcal P_n$ such that $C_{\pi}=C_{\lambda}$ and a representation $\mu_{\lambda}:\pi_1(\Sigma)\rightarrow \mathfrak {G}_{\pi}$. 

Suppose $\pi=(\mu_1^{d_1},\cdots, \mu_s^{d_s})\in \mathcal P_n$. The Higgs bundle $(E,\phi)^{\pi}$ which corresponds to $(\tau_{\pi}\circ j)\cdot\mu_{\pi}$ is a direct sum of Higgs bundles $(E_i, \phi_i)(1\leq i\leq k)$ where each $E_i$ is the tensor product of $Sym^{\mu_i-1}(K^{\frac{1}{2}}\oplus K^{-\frac{1}{2}}, \begin{pmatrix}0&0\\1&0\end{pmatrix})$ and $(V_i,0)$ for $V_i$  a polystable holomorphic vector bundle of rank $d_i$ satisfying $\prod_{i=1}^s\det(V_i)^{\mu_i}=\mathcal O$. 

We almost finish the proof except proving $\pi=\lambda$. Note that $C_{\lambda}=C_{\pi}$ does not imply $\lambda=\pi$, for example, see Remark \ref{PartitionLength}. The Higgs bundle $(E,\phi)^{\pi}$ is clearly of Jordan type $\pi$ at every point. By assumption, the Higgs bundle $(E,\phi)$ is of Jordan type at most $\lambda$, then $\pi\leq \lambda$. Since $C_{\lambda}=C_{\pi}$, it follows from Lemma \ref{OrderOfPartition} that $\pi=\lambda. $
\end{proof}

\begin{rem}\label{PartitionLength}
Even if two partitions $\lambda, \mu\in \mathcal P_n$ satsify $C_{\lambda}=C_{\mu}$, it is not necessarily that $\lambda=\mu$. For example, $\lambda=(5,5,1), \mu=(6,3,2)\in \mathcal P_{11}$ satisfy $C_{\lambda}=C_{\mu}=\frac{1}{20}$. So the length spectrum itself does not detect the Jordan type easily. 
\end{rem}

\begin{cor} For a polystable complex variation of Hodge structures $(E,\phi)$ over a Riemann surface $\Sigma$ of type $\mathcal R$, the corresponding representation $\rho$ satisifes either its length spectrum is strictly dominated by the one of $\tau_{\mathcal R^t}\circ j_\Sigma$ or it has the same length spectrum as $\tau_{\mathcal R^t}\circ j_\Sigma$.
\end{cor}

\section{Applications to holomorphic sectional curvature of period domain and Calabi-Yau moduli}\label{PeriodDomainModuli}

In this section, we discuss how the algebraic estimate for matrices of type $\mathcal R\in C_n$ applies to the study of curvature estimates of the period domain and the deformation space of Calabi-Yau manifolds. 

\subsection{Polarized variation of Hodge structure and period domain}\label{PeriodDomain}
Let us recall the preliminaries on variation of Hodge structures. For more details, one may refer to \cite{CarlsonPeters} \cite{Lu99}.

Let $X$ be a compact K\"ahler manifold of dimension $n$. A $(1,1)$-form $\omega$ is called a polarization of $X$ if $[\omega]$ is the first Chern class of an ample line bundle over $X$. The pair $(X,\omega)$ is called a polarized algebraic variety. Using the form $\omega$, one can define the $k$-th primitive cohomology $P^k(X,\mathbb C)\subset H^k(X,\mathbb C)$. 

Let $H_{\mathbb Z}=P^k(X,\mathbb C)\cap H^k(X,\mathbb Z)$ and $H^{p,q}=P^k(X,\mathbb C)\cap H^{p,q}(X)$ for $0\leq p,q\leq k$. Then we have $H=H_{\mathbb Z}\otimes \mathbb C=\Sigma_p H^{p,q}$ and $H^{p,q}=\overline{H^{q,p}}$. We call $\{H^{p,q}\}$ the Hodge decomposition of $H$. 
Equivalently, we may define a filtration of $H$ by $0\subset F^k\subset F^{k-1}\subset \cdots\subset F^0\subset H$, such that $H=F^p\oplus \overline{F^{n-p+1}}$ and $H^{p,q}=F^{p}\cap \overline{F^q}$. The set $\{H^{p,q}\}$ and $\{F^p\}$ are equivalent in describing the Hodge decomposition of $H$. Suppose $Q$ is the bilinear form on $H$ induced by the cup product 
$$Q(\phi,\psi)=(-1)^{\frac{k(k-1)}{2}}\int_X\phi\wedge\psi\wedge\omega^{n-k},\quad \phi,\psi\in H.$$ $Q$ is a nondegerate quadratic form, and is skew-symmetric if $n$ is odd and symmetric if $n$ is even, satisfying the two Hodge-Riemann relations:\\
(*) $Q(H^{p,q},H^{p',q'})=0$ unless $p'=n-p, q'=n-q$;\\
(**) $b(\cdot,\cdot)=Q(i^{p-q}\cdot,\bar\cdot)$ is a Hermitian inner product on $H^{p,k-p}$ for each $p$.
\begin{df}
A polarized Hodge structure of weight $n$ consists of $\{H_{\mathbb Z},F^p,Q\}$, is given by a lattice $H_{\mathbb Z}$, a filtration of $H=H_{\mathbb Z}\otimes \mathbb C$: $0\subset F^n\subset F^{n-1}\subset \cdots\subset F^0\subset H$, such that $H=F^p\oplus \overline{F^{n-p+1}}$ together with a bilinear form $Q: H_{\mathbb Z}\otimes H_{\mathbb Z}\rightarrow \mathbb Z$ which is skew-symmetric if $n$ is odd and symmetric if $n$ is even, satisfying the two Hodge-Riemann relations.
\end{df}
 
\begin{df}
The space $\mathcal D=\mathcal D(H, Q, k, \{h^{p,q}\})$ consisting of all Hodge structures of weight $k$ with fixed dimension $h^{p,q}$ of $H^{p,q}$, polarized by $Q$, is called the period domain.
\end{df}

The period domain $\mathcal D$ can be written as the homogeneous space $\mathcal D=G/V$ where $G=Aut(H_{\mathbb R},Q)\cap SL(H_{\mathbb R})$ for $H_{\mathbb R}=H_{\mathbb Z}\otimes \mathbb R$ and $V$ is the compact subgroup of $G$ fixing a reference filtration $F_0$. For odd weight $k=2m+1,$ $$G=Sp(n,\mathbb R), \dim H=2n, V=\prod_{p\leq m}U(h^{p,q});$$ for even weight $k=2m$, $$G=SO(s, t), s=\sum_{p~\text{even}}h^{p,q}, t=\sum_{p~\text{odd}}h^{p,q}, V=\prod_{p<m}U(h^{p,q})\times SO(h^{m,m}).$$ 

Let $Lie(G_{\mathbb C})=\mathfrak g_{\mathbb C}=\{X\in \mathfrak{sl}(H):Q(Xu,v)+Q(u,Xv)=0, \forall u,v\in H\}, Lie(G)=\mathfrak g=\mathfrak g_{\mathbb C}\cap \mathfrak{gl}(H_{\mathbb R})$. We have an $Ad(V)$-invariant splitting $\mathfrak g=\mathfrak v\oplus \mathfrak m$, so is the complexification $\mathfrak g_{\mathbb C}=\mathfrak v_{\mathbb C}+\mathfrak m_{\mathbb C}$. The Lie algebra $\mathfrak g_{\mathbb C}$ has a Hodge decomposition 
$$\mathfrak g_{\mathbb C}=\oplus_p \mathfrak g^{p,-p},\quad\mathfrak g^{p,-p}=\{\xi\in \mathfrak g_{\mathbb C}|\xi H^{r,s}\subset H^{r+p,s-p}\}.$$
And $\mathfrak v_{\mathbb C}=\mathfrak g^{0,0}, \mathfrak m_{\mathbb C}=\mathfrak m^+\oplus \mathfrak m^-$ where 
$\mathfrak m^+=\oplus_{p>0}\mathfrak g^{p,-p}, \mathfrak m^-=\oplus_{p<0}\mathfrak g^{p,-p}.$ $\mathcal D$ has a natural complex structure induced by $J$ acting on $\mathfrak m^{\pm}$ by $\pm i$ and the holomorphic tangent bundle of $\mathcal D$ is $T_{\mathcal D}^{hol}=G\times_V\mathfrak m^-$ as an associated bundle of the principal $V$-bundle $G\rightarrow \mathcal D$. 

Let $U$ be an open neighborhood of the universal deformation space of $X$. Assume that $U$ is smooth. Then for each $X'$ near $X$, we have an isomorphism $H^k(X',\mathbb C)=H^k(X,\mathbb C)$. Under this isomorphism, $\{P^k(X',\mathbb C)\cap H^{p,q}(X')\}_{p+q=k}$ can be considered as a point of $\mathcal D$. A polarized variation of Hodge structures is equivalent to the map
$$\mathcal P: U\rightarrow \mathcal D,\quad X'\rightarrow \{P^k(X',\mathbb C)\cap H^{p,q}(X')\}_{p+q=k},$$
called \textit{Griffiths' period map}. Griffiths showws that the period map $\mathcal P$ is an immersion, holomorphic and transversal, which is
\[d\mathcal P: TU\rightarrow G\times_V\mathfrak g^{-1,1}.\]
We then denote the horizontal distribution $T^h\mathcal D=G\times_V \mathfrak g^{-1,1}.$

The Cartan involution $\theta$ on $\mathfrak g_{\mathbb C}$ is 
\begin{equation}\label{CartanInvolution}\theta(X)=(-1)^p \bar X\quad \text{for $X\in\mathfrak g^{-p,p}$}.\end{equation}  The Cartan involution induces a $G$-invariant Hermitian metric $h$ on $T_{\mathcal D}^{\mathbb C}$: for $\xi=[g,X], \eta=[g,Y]\in T_{gV}^{hol}\mathcal D$, 
\begin{equation}\label{MetricPeriodDomain}h(\xi,\eta):=-B_0(X,\theta(Y)),\end{equation} where $B_0$ is the trace form as the rescaled Killing form on $\mathfrak g_{\mathbb C}$. Since $G$ is simple, any other $G$-invariant metric on $\mathcal D$ is a constant multiple of $h$.

\begin{prop}(Griffiths-Schmid \cite{GriffithsSchmid})
The $G$-invariant Hermitian metric $h$ on $\mathcal D$ has holomorphic sectional curvatures in the directions of $T^hD$ are negative and bounded away from zero.
\end{prop}
Here we give an effective estimate of holomorphic sectional curvatures of $T^h\mathcal D$.

\begin{thm}\label{PeriodDomainCurvature}
Suppose $\mathcal D=\mathcal D(H, Q, k, \{h^{p,q}\})$ where the composition $\mathcal R=(h^{k,0},h^{k-1,1},\cdots, h^{0,k})\in \mathcal C_m$ for $m={\sum_{p=0}^kh^{p,q}}$. Then the $G$-invariant Hermitian metric $h$ on $\mathcal D$ has holomorphic sectional curvature in the direction $\xi\in T^h\mathcal D$ satisfying 
\[K(\xi)\leq -C_{\mathcal R^t}.\]
where $\mathcal R^t$ is the conjugate partition of $\mathcal R$.

Moreover, the equality can be achieved in some direction $\xi$.
\end{thm}
\begin{proof}
We mainly adopt the language in Chapter 12 and 13 of \cite{CarlsonPeters}.
Let $\omega$ be the Maurer-Cartan $1$-form on $G$ and $\omega^{\mathfrak v}, \omega^{\mathfrak m}$ be its $\mathfrak v$-part, $\mathfrak m$-part respectively. Then on the principal $V$-bundle $G\rightarrow \mathcal D$, the canonical connection is $\omega^{\mathfrak v}$ and the curvature is $\Omega=-\frac{1}{2}[\omega^{\mathfrak m}, \omega^{\mathfrak m}]^{\mathfrak v}$. So for $X, Y\in \mathfrak m$ and $\xi=[g,X],\eta=[g,Y]$, then $\Omega(\xi, \eta)=-[X,Y]^{\mathfrak v}.$ 

Using the representation $Ad: V\rightarrow GL(\mathfrak m_{\mathbb C})$, the canonical connection on $G\rightarrow \mathcal D$ induced the Levi-Civita connection on $T_{\mathcal D}^{\mathbb C}=G\times_V\mathfrak m_{\mathbb C}$ with respect to the Hermitian metric $h$. The curvature $\Omega\in \Omega^2(G, \mathfrak b)$ on $G$ and the curvature $R$ on $T_{\mathcal D}^{\mathbb C}$ are related by the adjoint representation $ad:\mathfrak v_{\mathbb C}\rightarrow End(\mathfrak m_{\mathbb C})$ as follows: for $X, Y, Z,W\in \mathfrak m_{\mathbb C}$ and $\xi=[g, X], \eta=[g, Y], \mu=[g,Z], \nu=[g,W]\in T_{gV}^{\mathbb C}\mathcal D$,
\begin{eqnarray*}
&&R(\xi,\eta)\mu=[g, ad(\Omega(\xi, \eta))Z]=[g, -[[X,Y]^{\mathfrak v_{\mathbb C}},Z]],\\
&&h(R(\xi,\eta)\mu, \nu)=B_0([[X, Y]^{\mathfrak v_{\mathbb C}}, Z], \theta(W)).
\end{eqnarray*}

The holomorphic sectional curvature in the direction $\xi$ is 
$K(\xi)=\frac{h(R(\xi, \bar\xi)\xi, \xi)}{h(\xi,\xi)h(\bar\xi,\bar\xi)}$. Therefore, for $\xi=[g, X]\in T_{gV}^h\mathcal D=G\times_V \mathfrak g^{-1,1}$,\begin{eqnarray*}
 K(\xi)&=&\frac{h(R(\xi, \bar\xi)\xi, \xi)}{h(\xi,\xi)h(\bar\xi,\bar\xi)}\\
 &&\text{$\bar\xi=[g,\bar X]=[g, -\theta(X)]$ from Equation (\ref{CartanInvolution})}\\
&=&\frac{B_0([[X, \theta(X)], X], \theta(X))}{B_0(X,\theta(X))B_0(\theta(X),X)}\\
&=&\frac{B_0([X, \theta(X)],[X, \theta(X)])}{B_0(X,\theta(X))^2}.
 \end{eqnarray*}

Recall that a Hodge frame consists of a $b$-unitary basis for $H_C$ such that 
\begin{itemize}
\item it provides bases for all Hodge summands $H^{p,q}$;
\item the conjugate of the resulting basis for $H^{p,q}$ forms part of the Hodge frame: it is a basis for $H^{q,p}.$
\end{itemize}
With respect to  a Hodge frame, $\theta(X)=-X^*=-\bar X^T$ and using that $B_0$ is the trace form, we have $K(\xi)=-\frac{||[X, X^*]||^2}{||X||^4}$.
Since $X\in \mathfrak g^{-1,1}$, it is a matrix of type $(h^{k,0}, h^{k-1,1},\cdots, h^{0,k})\in \mathcal C_m$ for $m=\sum_{p=0}^kh^{k-p,p}$ as in Definition \ref{TypeR}. 
Then by Proposition \ref{CompositionEstimate}, we have $K(\xi)\leq -C_{\mathcal R^t}.$

The rest is to show that the equality can be achieved. With respect to a Hodge frame, if we choose $X$ as the form of matrix in Proposition \ref{CompositionEstimate} which achieves the equality, then $K(\xi)=-K(X)=-C_{\mathcal R^t}$. The remaining step is to show that such a $X$ lies in $\mathfrak g_{\mathbb C}$. Equivalently, $$\mathfrak g_{\mathbb C}=\{X\in \mathfrak{gl}(H_{\mathbb C}); h(Xu,\overline v)=h(u,\overline{Xv})\}.$$ Denote the base in the Hodge frame and in $H^{p,q}$ is $\{u^{p,q}_i\}_{i=1}^{h^{p,q}}$. By the definition of a Hodge frame, $\overline{u^{q,p}_i}=u^{p,q}_i$. Then $Xu^{p,q}_i=b_i^{p,q}u^{p-1,q+1}_i$ for some $b_i^{p,q}\in \mathbb C$. So 
\begin{eqnarray}
&&b(Xu^{p,q}_i,\overline{u^{p',q'}_j})=b_i^{p,q}\delta_{p-1,q'}\\
&&b(u^{p,q}_i,\overline{Xu^{p',q'}_i})=\overline{b_i^{p',q'}}\delta_{p,q'+1}=\overline{b_i^{p',q'}}\delta_{p-1,q'}
\end{eqnarray}
It is enough to check $b_i^{p,q}=\overline{b_i^{q+1,p-1}}$, which holds for the following reason. If we adopt notations in Proposition \ref{CompositionEstimate}, $b_i^{p,q}=a_i^{k+1-p}\in \mathbb R$. Since $h^{p,q}=h^{q+1,p-1}$, the composition $(h^{k,0},h^{k-1,1},\cdots, h^{0,k})$ has a symmetry. So from Remark \ref{Symmetry}, $a_i^{k+1-p}=a_i^{p}\in \mathbb R$. Then $b_i^{p,q}=a_i^{k+1-p}=a_i^{p}=b_i^{k+1-p, p-1}=b_i^{q+1,p-1}\in \mathbb R$ and thus $X\in\mathfrak g_{\mathbb C}$.
 \end{proof}

A polarized variation of Hodge structures induces a Higgs bundle together with a harmonic metric as follows. For a family $\pi:\mathcal X\rightarrow M$, consider the local system $\mathbb H:=R^k\pi_*\mathbb C$ determined by the presheaf $U\mapsto P^k(\pi^{-1}(U), \mathbb C)$. Then $\mathcal H=\mathbb H\otimes\mathcal O_M$ is a holomorphic vector bundle which carries a natural flat connection $\nabla$, called the Gauss-Manin connection. The subbundles $\mathcal F^p$ formed by the filtrations $F^p$ are holomorphic, called \textit{the Hodge bundles}. By Griffiths transversality, $\nabla$ takes $\mathcal F^p$ to $\Omega^1(\mathcal F^{p-1})$. \\
(1) the holomorphic vector bundle $E$ is the graded bundle of the holomorphic filtration, that is 
$$E:=\sum_{p}\mathcal F^p/\mathcal F^{p+1}=\sum\limits_{p}\mathcal H^{p,q}.$$ 
(2) The flat connection $\nabla$ induces a holomorphic map $\phi$ taking $\mathcal H^{p,q}$ to $\Omega^1(\mathcal H^{p-1,q+1})$, the Higgs field, satisfying $\phi\wedge \phi=0$.\\
(3) The Hermitian metric $b$ in the polarized variation of Hodge structure induces a harmonic metric $h$ on the Higgs bundle.

In this way, we obtain $(E, \phi, h)$, a Higgs bundle together with a harmonic metric.  Note that $\mathcal H$ and $E$ are isomorphic as $C^{\infty}$-bundles and $\nabla=\nabla_{\bar\partial_{E}, h}+\phi+\phi^{*_h}$, where $\nabla_{\bar\partial_{E},h}$ is the Chern connection uniquely determined by $\bar\partial_E$ and $h$. Set $K$ is the connected component of the maximal subgroup of $G$ containing $V$ and $G/K$ is a Riemannian symmetric space. There is a natural isometric fibration $p:\mathcal D=G/V\rightarrow G/K$ sending $\{H^{p,q}\}$ to the subspace $H^{k,0}+H^{k-2,2}+\cdots$. One can see that the composed map $f:=p\circ \mathcal P:U\rightarrow G/K\rightarrow SL(n,\mathbb C)/SU(n)$ is a pluriharmonic map to $SL(n,\mathbb C)/SU(n)$ which corresponds to the Higgs bundle $(E, \phi, h)$ constructed from above by the nonabelian Hodge correspondence. On $SL(n,\mathbb C)/SU(n)$, let $g_{SL(n,\mathbb C)/SU(n)}$ be the invarian metric induced by the trace form on $sl(n,\mathbb C)$, so $f^*g_{SL(n,\mathbb C)/SU(n)}=\mathcal P^*g_{\mathcal D}$. Similar to the formula (\ref{PullbackMetric}) on the Riemann surface case, the pullback metric $f^*g_{SL(n,\mathbb C)/SU(n)}=tr(\phi\phi^{*_h})$ since $\phi$ is nilpotent.  

\subsection{Calabi-Yau moduli, Hodge metric  and Weil-Petersson metric}\label{CalabiYauModuli}
A polarized Calabi-Yau manifold is a pair $(X,\omega)$ of a compact algebraic manifold $X$ of dimension $m$ with vanishing first Chern class and a K\"ahler form $\omega\in H^2(X,\mathbb{Z})$. The universal deformation space $\mathcal{M}$ of polarized Cababi-Yau $m$-manifold $(X,\omega)$ is smooth, shown by Tian \cite{Tian}. The tangent space $T_{X\prime}\mathcal{M}_X$ of $\mathcal{M}_X$ at $X^{\prime}$ can be identified with $H^1(X^{\prime}, T_{X^{\prime}})$. Suppose $n=\dim H^1(X,T_X)=\dim H^{1,m-1}$. Let $G_{\mathbb Z}:=\{g\in G: gH_{\mathbb Z}=H_{\mathbb Z}\}$. Then we have the period map $\mathcal P:\mathcal M_X\rightarrow \Gamma\backslash\mathcal D$ where $\Gamma<G_{\mathbb Z}$ is the monodromy group and $\mathcal D=\mathcal D(H,Q,n,\{h^{p,m-p}\})$ where $h^{p,m-p}$ is the dimension of the $(p,m-p)$-primitive cohomology. So $h^{m,0}=h^{0,m}=1, h^{m-1,1}=h^{1,m-1}=n.$

1. The \textit{Hodge metric} on $\mathcal M_X$ was first defined in Lu \cite{Lu} as the pullback metric $\mathcal P^*h$ on $\mathcal M_X$ by the period map $\mathcal P:\mathcal M_X\rightarrow \Gamma\backslash\mathcal D$. Moreover, Lu \cite{Lu} showed that the Hodge metric is a K\"ahler metric; the bisectional curvatures of the Hodge metric are nonpositive; the Ricci curvature of the Hodge metric is bounded above by a negative constant. 

2. Let $\mathcal F^m$ be the first Hodge bundle over $\mathcal M$ formed by $H^{m,0}(X)$, the Weil-Petersson metric on $\mathcal M_X$ is defined as $$\omega_{WP}=c_1(F^m)=-\frac{i}{2\pi}\partial\bar\partial Q(\Omega,\bar\Omega),$$
where $\Omega$ is a local holomorphic nonzero section of $\mathcal F^n$ and $Q$ is the polarization.

3. The relation between the Weil-Petersson metric and the Hodge metric on $\mathcal M_X$ is clear.  Let $\omega_H$ be the K\"ahler form of the Hodge metric $H$.
\begin{prop} \label{RelationTwoMetrics}
(1) In the case of twofold, $\omega_H=2\omega_{WP}$;\\
(2) (Lu \cite{Lu}) In the case of threefold, $\omega_H=(n+3)\omega_{WP}+Ric(\omega_{WP})$;\\
(3) (Lu-Sun \cite{LuSun}) In the case of fourfold, $\omega_H=2(n+2)\omega_{WP}+2Ric(\omega_{WP})$;\\
(4) (\cite{LuSun}, \cite{LuSun2}) In the case of higher dimension, we only have the inequality $\omega_H\geq 2(n+2)\omega_{WP}+2Ric(\omega_{WP})\geq 2\omega_{WP}.$
\end{prop}

\begin{thm}\label{HodgeMetric} The Hodge metric over the deformation space of Calabi-Yau $k$-fold is of holomorphic sectional curvature  bounded above by $-C_{\mathcal R^t}$ where $\mathcal R=(h^{k,0}, h^{k-1,1}, \cdots, h^{0,k})$.\end{thm}
\begin{proof}
We have two proofs which are essentially identical. 

First, since locally the image of $\mathcal P$ is a complex submanifold inside $\mathcal D$ and tangential to the horizontal distribution $T^h\mathcal D$, the holomorphic sectional curvatures of the Hodge metric on $\mathcal M_X$ are no more than the ones of $T^h\mathcal D$. So the statement follows from Theorem \ref{PeriodDomainCurvature}. 

Or we use the Higgs bundle description of the period map associated to the moduli space $\mathcal M_X$ in the end of Section \ref{PeriodDomain} and obtain a complex variation Hodge structure of type $\mathcal R$. Since the Hodge metric $tr(\phi\phi^*)$ defined from the Higgs bundle viewpoint coincides with the one $\mathcal P^*g_{\mathcal D}$ from the period map viewpoint, then the estimate follows from Corollary \ref{CVHSCurvature}. 
\end{proof}
\begin{rem}
Our results are sharp in the abstract version.
\end{rem} 
\begin{cor}
The Hodge metric over the deformation space of Calabi-Yau threefolds is of holomorphic sectional curvature bounded above by $-\frac{2}{n+9}$.
\end{cor}
\begin{proof}
Since in case $k=3$, the conjugate  partition $\mathcal R^t$ of $\mathcal R=(1,n,n,1)$ is $(4,2^{n-1})$. We have
\begin{align*}
&\text{Composition $\mathcal R$:}&(1, n=4, n=4, 1)\\
&\text{Generalized Young Diagram:}&\yng(1,4,4,1)\\
&\text{Conjugate Set Partition:}&(\{4\}, \{2\}, \{2\}, \{2\})\\
&\text{Conjugate Partition $\mathcal R^t$:}&(4, 2^{n-1}=2^3)
\end{align*}

Note that $$C_{\mathcal R^t}=\frac{12}{4(4^2-1)+(n-1)\cdot 2(2^2-1)}=\frac{12}{60+6(n-1)}=\frac{2}{n+9}.$$
Then the corollary follows from Theorem \ref{HodgeMetric}.
\end{proof}

\begin{cor}
The Hodge metric over the deformation space of Calabi-Yau fourfolds is of holomorphic sectional curvature with bounded above by $-\frac{1}{2(\min\{a,n\}+4)}$,  where $a=h^{2,2}$.
\end{cor}
\begin{proof}
We consider complex variation of Hodge structures of type $\mathcal R=(1,n,a,n,1)$, where $a=h^{2,2}$. In case $a\leq n$, the conjugate partition $\mathcal R^t$ is $\lambda_1=(5,3^{a-1},1^{2n-2a})$. In case $a\geq n$, the conjugate partition $\mathcal R^t$ is $\lambda_2=(5,3^{n-1},1^{a-n})$. 
\begin{tiny}
\begin{align*}
&\text{Composition $\mathcal R$:}&(1, n=5, a=3, n=5, 1)\quad&\quad(1, n=3, a=5, n=3,1)\\
&\text{Generalized Young Diagram:}&\yng(1,5,3,5,1)\quad&\quad\quad\quad\yng(1,3,5,3,1)\\
&\text{Conjugate Set Partition:}&(\{5\}, \{3\}, \{3\}, \{1, 1\}, \{1,1\})\quad&\quad (\{5\}, \{3\}, \{3\},\{1\},\{1\})\\
&\text{Conjugate Partition $\mathcal R^t$:}&(5, 3^2=3^{a-1}, 1^4=1^{2n-2a})\quad&\quad(5, 3^2=3^{n-1}, 1^3=1^{a-n})
\end{align*}
\end{tiny}
Note that $$C_{\lambda_1}=\frac{12}{5(5^2-1)+(a-1)3(3^2-1)}=\frac{1}{10+2(a-1)}=\frac{1}{2(a+4)},$$  $$C_{\lambda_2}=\frac{12}{5(5^2-1)+(n-1)3(3^2-1)}=\frac{1}{10+2(n-1)}=\frac{1}{2(n+4)}.$$
Then the corollary follows from Theorem \ref{HodgeMetric}.
\end{proof}

In the case of Calabi-Yau $k$-manifolds for $k\geq 5$, our estimates are new. 
\begin{cor}
The Hodge metric over the deformation space of Calabi-Yau 5-folds is of holomorphic sectional curvature bounded above by\\
$-\frac{2}{(9\min\{a,n\}+a+25)}$, where $a=h^{2,3}$.
\end{cor}
\begin{proof}
We consider complex variation of Hodge structures of type $\mathcal R=(1,n,a,a,n,1)$, where $a=h^{2,3}=h^{3,2}$. 
In case $a\leq n$, the conjugate partition $\mathcal R^t$ is $\lambda_1=(6,4^{a-1},1^{2n-2a})$. In case $a\geq n$, the conjugate partition $\mathcal R^t$ is $\lambda_2=(6,4^{n-1},2^{a-n})$. 
\begin{tiny}
\begin{align*}
&\text{Composition $\mathcal R$:}&(1, n=4, a=2, a=2, n=4, 1)\quad&\quad(1, n=2, a=4,a=4, n=2,1)\\
&\text{Generalized Young Diagram:}&\yng(1,4,2,2,4,1)\quad&\quad\quad\quad\quad\yng(1,2,4,4,2,1)\\
&\text{Conjugate Set Partition:}&(\{6\}, \{4\}, \{1, 1\}, \{1,1\})\quad&\quad (\{6\}, \{4\},\{2\},\{2\})\\
&\text{Conjugate Partition $\mathcal R^t$:}&(6, 4=4^{a-1}, 1^4=1^{2n-2a})\quad&\quad(6, 4=4^{n-1}, 2^2=2^{a-n})
\end{align*}
\end{tiny}

Note that $$C_{\lambda_1}=\frac{12}{6(6^2-1)+(a-1)4(4^2-1)}=\frac{2}{35+10(a-1)}=\frac{2}{10a+25},$$  $$C_{\lambda_2}=\frac{12}{6(6^2-1)+(n-1)4(4^2-1)+(a-n)2(2^2-1)}=\frac{2}{9n+a+25}.$$
Then the corollary follows from Theorem \ref{HodgeMetric}.
\end{proof}

\begin{cor}The Hodge metric over the deformation space of Calabi-Yau $k$-fold is of holomorphic sectional curvature  bounded above by $$-\frac{2}{k^2+n(k-2)^2+\sum\limits_{p=2}^{[\frac{k}{2}]}h^{p,k-p}\cdot (k-2p)^2}.$$ \end{cor}
\begin{proof}
For $k=2m+1$ is odd, we can choose a permutation of $\mathcal R$ as $\mathcal R_1=(h^{m,m+1}, h^{m+1,m}, h^{m-1,m+2},h^{m+2,m-1}, \cdots, h^{2,2m-1},h^{2m-1,2}, n,n,1,1)$. Then by Part (iv) of Lemma \ref{OrderOfPartition}, we have $C_{\mathcal R^t}\leq D_{\mathcal R_1}$. So $\mathcal R_1=(r_1,\cdots, r_{2m+2})$ where $r_{2p+1}=r_{2p+2}=h^{m-p,m+p+1}$ for $0\leq p\leq m$.  Using the formula (\ref{ConstantD}), we have
\begin{eqnarray*}
D_{\mathcal R_1}=\frac{4}{\sum\limits_{p=1}^{2m+2}r_p\cdot p(p-1)}
=\frac{4}{\sum\limits_{p=0}^{m}h^{m-p,m+p+1}\cdot [(2p+1)2p+(2p+2)(2p+1)]}\\
=\frac{2}{\sum\limits_{p=0}^{m}h^{m-p,m+p+1}\cdot (2p+1)^2}
=\frac{2}{k^2+n(k-2)^2+\sum\limits_{p=2}^m h^{p,k-p}\cdot (k-2p)^2}.
\end{eqnarray*}

For $k=2m$ is even, we can choose a permutation of $\mathcal R$ as \\
$\mathcal R_1=(h^{m,m}, h^{m-1,m+1}, h^{m+1,m-1},\cdots,h^{2,2m-2},h^{2m-2,2}, n,n,1,1)$. The rest is similar to the odd case.

Then the corollary follows from Theorem \ref{HodgeMetric}.
\end{proof}

\end{document}